\documentclass[12pt]{amsart}

\usepackage[headings]{fullpage}
\usepackage{amssymb,epic,eepic,epsfig,amsbsy,amsmath,amscd,color,graphicx,subcaption,texdraw,bbm,url}
\usepackage[all]{xy}
\usepackage{mathrsfs}
\usepackage{standalone}
\usepackage[normalem]{ulem}
\usepackage{comment}
\usepackage{soul}
\usepackage[final]{pdfpages}
\usepackage{tikz}
\usetikzlibrary{cd,positioning,knots,patterns,hobby,math,decorations}

\theoremstyle{plain}
\newtheorem{theorem}{Theorem}[section]
\newtheorem{thm}{Theorem}
\newtheorem{lemma}[theorem]{Lemma}
\newtheorem{corollary}[theorem]{Corollary}
\newtheorem{proposition}[theorem]{Proposition}
\newtheorem{conjecture}{Conjecture}
\newtheorem{question}{Question}

\newtheorem{definition}{Definition}

\theoremstyle{definition}

\newtheorem{remark}[theorem]{Remark}
\newtheorem{example}[theorem]{Example}
\newcommand{\bcon}{\begin{conjecture}}
\newcommand{\econ}{\end{conjecture}}
\newcommand{\bcor}{\begin{corollary}}
\newcommand{\ecor}{\end{corollary}}
\newcommand{\bdf}{\begin{definition}}
\newcommand{\edf}{\end{definition}}
\newcommand{\benu}{\begin{enumerate}}
\newcommand{\eenu}{\end{enumerate}}
\newcommand{\beq}{\begin{equation}}
\newcommand{\eeq}{\end{equation}}
\newcommand{\bexa}{\begin{example}}
\newcommand{\eexa}{\end{example}}
\newcommand{\bexe}{\begin{exercise}}
\newcommand{\eexe}{\end{exercise}}
\newcommand{\bfac}{\begin{fact}}
\newcommand{\efac}{\end{fact}}
\newcommand{\bite}{\begin{itemize}}
\newcommand{\eite}{\end{itemize}}
\newcommand{\blem}{\begin{lemma}}
\newcommand{\elem}{\end{lemma}}
\newcommand{\bmat}{\begin{pmatrix}}
\newcommand{\emat}{\end{pmatrix}}
\newcommand{\bprb}{\begin{problem}}
\newcommand{\eprb}{\end{problem}}
\newcommand{\bpro}{\begin{proposition}}
\newcommand{\epro}{\end{proposition}}

\newcommand{\bque}{\begin{question}}
\newcommand{\eque}{\end{question}}
\newcommand{\brem}{\begin{remark}}
\newcommand{\erem}{\end{remark}}
\newcommand{\bthm}{\begin{theorem}}
\newcommand{\ethm}{\end{theorem}}

\newcommand{\bpr}{\begin{proof}}
\newcommand{\epr}{\end{proof}}
\newcommand{\ignore}[1]{}

\newcommand{\no}[1]{}

\no{
\usepackage[utf8]{inputenc}
\usepackage{color}

\usepackage[letterpaper,total={6.5in,9in},top=1in, left=1in, right=1in, bottom=1in]{geometry}

\usepackage{fancyhdr}
\pagestyle{fancy}
\fancyhead{}
\fancyhead[CE]{{\bf {\scriptsize RHEA PALAK BAKSHI, THANG T. Q. LÊ, AND JÓZEF H. PRZYTYCKI}}}
\fancyhead[CO]{{\bf {\scriptsize KAUFFMAN BRACKET SKEIN MODULE OF THE CONNECTED SUM OF TWO SOLID TORI}}}
\fancyfoot{}
\fancyfoot[C]{{\scriptsize \thepage}}
\setlength{\footskip}{20pt}
\setlength\headheight{14pt}
\usepackage{libertine}
\usepackage[libertine]{newtxmath}
\usepackage{color,amsmath,mathtools,graphicx}
\usepackage{bbold}
\usepackage{epsfig,fancyhdr}
\usepackage{dsfont} 
\usepackage{graphicx,amsmath,mathtools,float}
\usepackage{epsfig,color,fancyhdr,setspace}
\usepackage{dsfont}
\usepackage{epstopdf}
\usepackage{import}
\usepackage{lineno}
\usepackage{stackrel,dsfont}
}
\usepackage[allcolors = blue,colorlinks]{hyperref}

\usepackage[abs]{overpic}

\no{
\usepackage{subcaption}
}
\usepackage{tikz-cd}
\usepackage{enumitem}
\counterwithin{figure}{section}
\usepackage[utf8]{inputenc}
\usepackage[T1]{fontenc}
\usepackage{ragged2e}
\usepackage{extarrows}

\no{
\newtheorem{theorem}{Theorem}[section]
\newtheorem{thm}{Theorem}
\newtheorem{lemma}[theorem]{Lemma}
\newtheorem{definition}[theorem]{Definition}
\newtheorem{example}[theorem]{Example}
\newtheorem{proposition}[theorem]{Proposition}
\newtheorem{remark}[theorem]{Remark}
\newtheorem{corollary}[theorem]{Corollary}
\newtheorem{conjecture}[theorem]{Conjecture}
\newtheorem{problem}[theorem]{Problem}

\newtheorem{exercise}[theorem]{Exercise}

\newtheorem{question}[theorem]{Question}
\newtheorem*{theorem*}{Theorem}
\newtheorem*{theorem**}{Theorem}

}

\newcommand\inclpdf[2]{
\raisebox{-.48 \height}{\includegraphics[height=#1]{draws/#2.pdf}}}

\no{
\newcommand{\bcon}{\begin{conjecture}}
\newcommand{\econ}{\end{conjecture}}
\newcommand{\bcor}{\begin{corollary}}
\newcommand{\ecor}{\end{corollary}}
\newcommand{\bdf}{\begin{definition}}
\newcommand{\edf}{\end{definition}}
\newcommand{\benu}{\begin{enumerate}}
\newcommand{\eenu}{\end{enumerate}}
\newcommand{\beq}{\begin{equation}}
\newcommand{\eeq}{\end{equation}}
\newcommand{\bexa}{\begin{example}}
\newcommand{\eexa}{\end{example}}
\newcommand{\bexe}{\begin{exercise}}
\newcommand{\eexe}{\end{exercise}}
\newcommand{\bfac}{\begin{fact}}
\newcommand{\efac}{\end{fact}}
\newcommand{\bite}{\begin{itemize}}
\newcommand{\eite}{\end{itemize}}
\newcommand{\blem}{\begin{lemma}}
\newcommand{\elem}{\end{lemma}}
\newcommand{\bmat}{\begin{pmatrix}}
\newcommand{\emat}{\end{pmatrix}}
\newcommand{\bprb}{\begin{problem}}
\newcommand{\eprb}{\end{problem}}
\newcommand{\bpro}{\begin{proposition}}
\newcommand{\epro}{\end{proposition}}

\newcommand{\bque}{\begin{question}}
\newcommand{\eque}{\end{question}}
\newcommand{\brem}{\begin{remark}}
\newcommand{\erem}{\end{remark}}
\newcommand{\bthm}{\begin{theorem}}
\newcommand{\ethm}{\end{theorem}}
\newcommand{\bpr}{\begin{proof}}
\newcommand{\epr}{\end{proof}}

}

\def\BZ{\mathbb Z}

\def\BQ{\mathbb Q}

\def\al{\alpha}

\def\be { \begin{equation} }
\def\ee { \end{equation} }

\def\SM{\cS(M)}

\def\id{\mathrm{id}}

\def\pM{\partial M}
\def\cS{\mathscr S}
\def\ot{\otimes}

\def\embed{\hookrightarrow}
\def\onto{\twoheadrightarrow}
\def\Zq{{\mathbb Z[q^{\pm1}]}}
\def\cK{{\mathcal K}}

\newcommand{\red}[1]{{\color{red}#1}}

\def\SH{{\cS(H_1 \# H_1)}}
\def\HH{ H_1 \ \# \ H_1}

\title{Kauffman bracket skein module of the connected sum of two solid tori}

\address{Department of Mathematics, University of California, Santa Barbara}
\email{{\rm rheapalak@math.ucsb.edu $|$ rheapalakbakshi@gmail.com}}
\author{ Rhea Palak Bakshi }
\address{Department of Mathematics, Georgia Institute of Technology, Atlanta, USA}
\email{{\rm letu@math.gatech.edu}}
\author{ Thang T. Q. L\^e}

\address{Department of Mathematics, The George Washington University, Washington DC, USA and Department of Mathematics, University of Gda\'{n}sk, Poland}
\email{{\rm przytyck@gwu.edu}}
\author{J\'ozef H. Przytycki}


\begin{document}
\begin{abstract}

We determine the structure of the Kauffman bracket skein module of the connected sum of two genus one handlebodies over the ring of Laurent polynomials $\mathbb Z[q^{\pm 1}]$, thereby proving a conjecture posed by the first and third authors. Our results lay the groundwork for computing the Kauffman bracket skein module of arbitrary connected sums over the ring $\mathbb Z[q^{\pm 1}]$. 

\end{abstract}

\maketitle

\section{Introduction}\label{introsectionh1h1}

\subsection{Kauffman bracket skein modules and algebras}  The Kauffman bracket skein module is an invariant of 
$3$-manifolds that has played a central role in low-dimensional topology and knot theory. Skein modules were introduced independently by the third author~\cite{smof3} and by Turaev~\cite{turaevsolidtorus,Turaev1}, and they provides a powerful algebraic framework for studying links embedded in 
$3$-manifolds and for connecting classical topology with quantum invariants. Given an oriented 
$3$-manifold 
$M$, possibly with non-empty boundary 
$\partial M$, the Kauffman bracket skein module, $\SM$ is defined as the module over the Laurent polynomial ring $R=\Zq$
 generated by isotopy classes of unoriented framed links in $M$, modulo the Kauffman bracket skein relations~\cite{tlnkauffman}. Skein modules serve as a natural bridge between topology, geometry, and quantum algebra. They are closely related to algebraic geometry via $SL(2,\mathbb{C})$ character varieties \cite{sl2cbullock, ps1}, hyperbolic geometry  via quantum Teichmüller spaces and quantum cluster algebras \cite{bw1,fokchekhov,lequantum,kashaev,mullerskein}, the AJ conjecture \cite{FGL,2bk}, and the Witten-Reshetikhin-Turaev $3$-manifold invariants and Topological Quantum Field Theories \cite{BHMV1}, to name a few.


\subsection{Calculations of skein modules}

The Kauffman bracket skein module (KBSM), over the ring $\Zq$,  has been fully calculated  only for a few classes of manifolds. These include lens spaces \cite{kbsmlens}, $S^1 \times S^2$ \cite{s1s2}, the Whitehead manifold \cite{whitehead}, $3$-manifolds obtained from integral surgery on the trefoil knot \cite{trefoilbullock}, the product of the pair of pants with $S^1$ \cite{pairofpantss1}, a family of prism manifolds \cite{prism}, $\mathbb RP^3 \ \# \ \mathbb RP^3$ \cite{rp3rp3}, and the exteriors of two-bridge knots and links \cite{2bk,2bl}. In \cite{connsum}, the third author stated a theorem about the structure of the KBSM of the connected sum of two handlebodies of arbitrary genera, over the ring $\mathbb{Z}[q^{\pm 1}]$. This ``theorem''  was disproved in \cite{counterhandle}  for connected sums involving handlebodies that are not both solid tori. However, the first and third authors conjectured that the statement about the structure of the KBSM over $\mathbb{Z}[q^{\pm 1}]$ remains valid when both summands are solid tori.

\def\hR{\hat R}
\def\fG{\mathcal{G}}
In this paper, we refine this conjecture and completely determine the structure of the Kauffman bracket skein module of $H_1 \ \# \ H_1$.
Note that this $3$-manifold is also the complement of the trivial link with two components. 
For a positive integer $k$, let
$ \{k\} = q^{2k} - q^{-2k} 
$. Define 
the Chebyshev polynomials of the second kind $S_n (x) \in \BZ[x]$ recursively by
\be 
S_0(x)=1, \ S_1(x)=x, \ S_n(x) = x S_{n-1}(x) - S_{n-2}(x).
\ee
Let $\hR= \BZ[x_1, x_2]$, the ring of polynomials in two formal variables $x_1$ and $x_2$.
The bigger ring $R[x_1, x_2, y]$ is an $\hR$-module. Consider its $\hR$-submodule
$ \fG\subset \BZ[x_1, x_2,y]$  spanned by $\{  G_n, n \ge 1\}$, where
$$ G_n =  \{n+1\}S_n(y) +  (-1)^{(n+1)} \{1\}S_n(x_1)S_n(x_2) \}.  $$

 Let $H_1$ be the solid torus. Our main result is the following. 

\begin{thm} \label{thm1}
Let the ground ring $R$ be $\Zq$. There is an explicit  $R$-linear isomorphism
$$ \phi: R[x_1, x_2, y]/\fG \xrightarrow {\cong} \cS(H_1 \# H_1) .  $$

\end{thm}
Roughly speaking,  $\phi(x_1^{n_1} x_2^{n_2} y^n)$ is the union of $n_1$ parallel copies of a core  $\overline x_1$ of the first copy of $H_1$, $n_2$ parallel copies of a core  $\overline x_2$ of the second copy of $H_2$, and $n$ parallel copies of a connected sum $\overline y$ of $\overline x_1$ and $\overline x_2$. The conjecture of the first and third author in \cite{counterhandle} does not describe the kernel $\fG$ explicitly, which we do. For details see Section \ref{sec.maintheorem}.

Our results will help in the study of the Kauffman bracket skein module of the connected sums of $3$-manifolds. The proof of our main result utilises a recently proved identity of Esebre and Gelca \cite{EG}. Furthermore, the stated skein modules of connected sums of $3$-manifolds at roots of unity exhibit exotic and strange properties, which are counterintuitive, see for example \cite{CL}.

The paper is structured as follows. In Section \ref{prelimsection}, we recall the Kauffman bracket skein module and its relative version. 
In Section \ref{sec.maintheorem}, we prove our main theorem. We also exhibit some torsion elements in the KBSM of $H_1 \ \# \ H_1$. 

\subsection{Acknowledgements}

The authors acknowledge support from the NSF grant DMS-2204001 for funding TTQL and JHP's research visits to UCSB. TTQL was partially supported by the NSF grant DMS-2203255. JHP was partially supported by the Simons Collaboration Grant 637794 and the CCAS Dean’s Research Chair award. The authors would like to thank R. Gelca for helpful discussions.

\section{Preliminaries}\label{prelimsection}

In this section, we discuss the definition of the Kauffman bracket skein module (KBSM) and algebra, along with their fundamental properties, including the behaviour of the KBSM under handle slides. Throughout the paper, the ground ring $R$ is the ring $\Zq$ of Laurent polynomials in the formal variable $q$ with integer coefficients. 

\def\SF{\cS (F)}
\def\bF{\bar F}
\def\KF{\SF}
\def\cC{\mathcal C}

\subsection{Skein modules} 
Let $M$ be an oriented 3-manifold, with possibly non-empty boundary. The Kauffman bracket skein module of $M$, denoted by $\SM$, is the $R$-module  freely spanned by all isotopy classes of unoriented framed links in $M$, including the empty link, subject to the following relations:
\begin{align}
\label{kbsr} \vcenter{\hbox{\includegraphics[scale=.07]{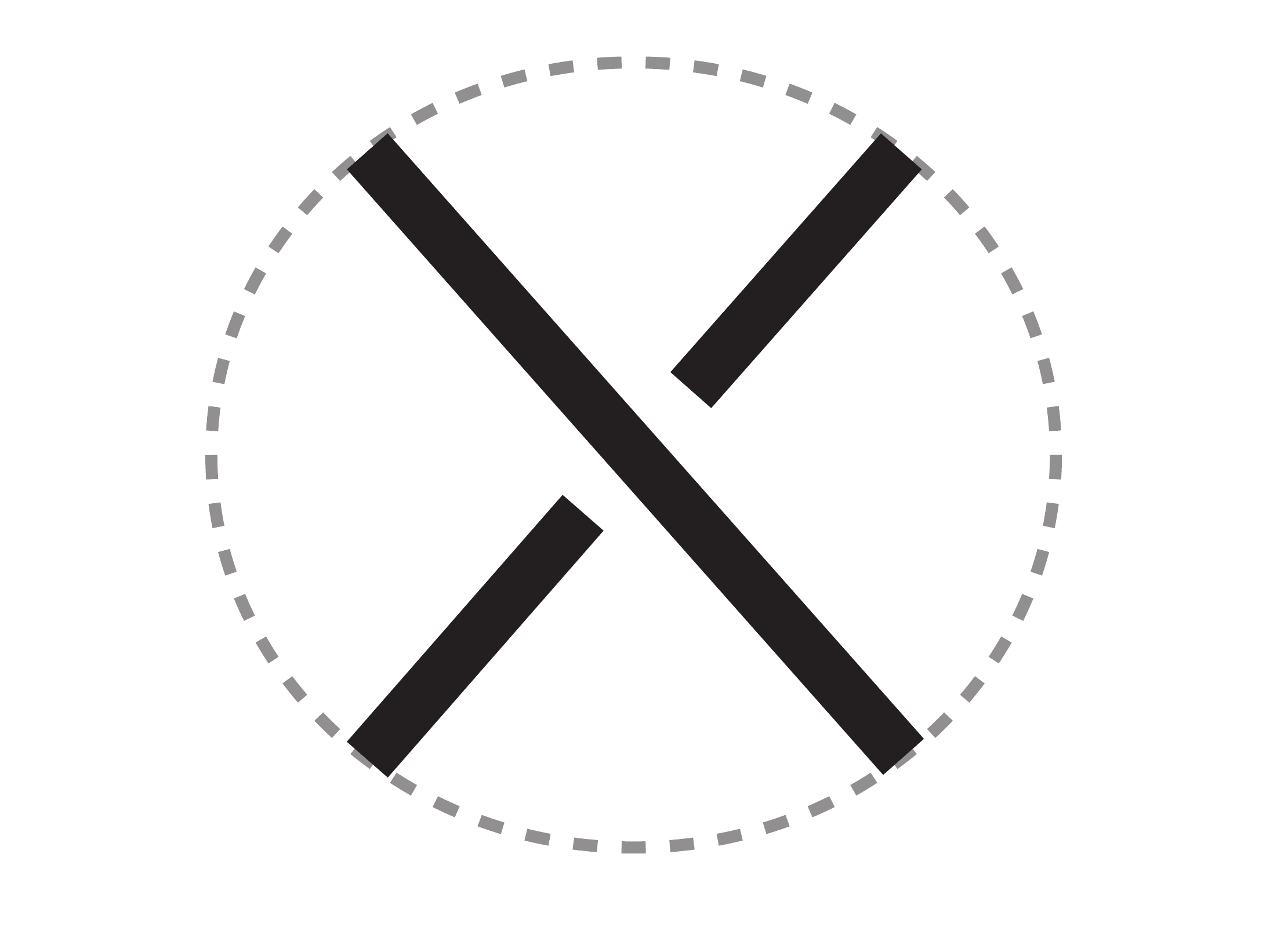}}} &- q \vcenter{\hbox{\includegraphics[scale=.07]{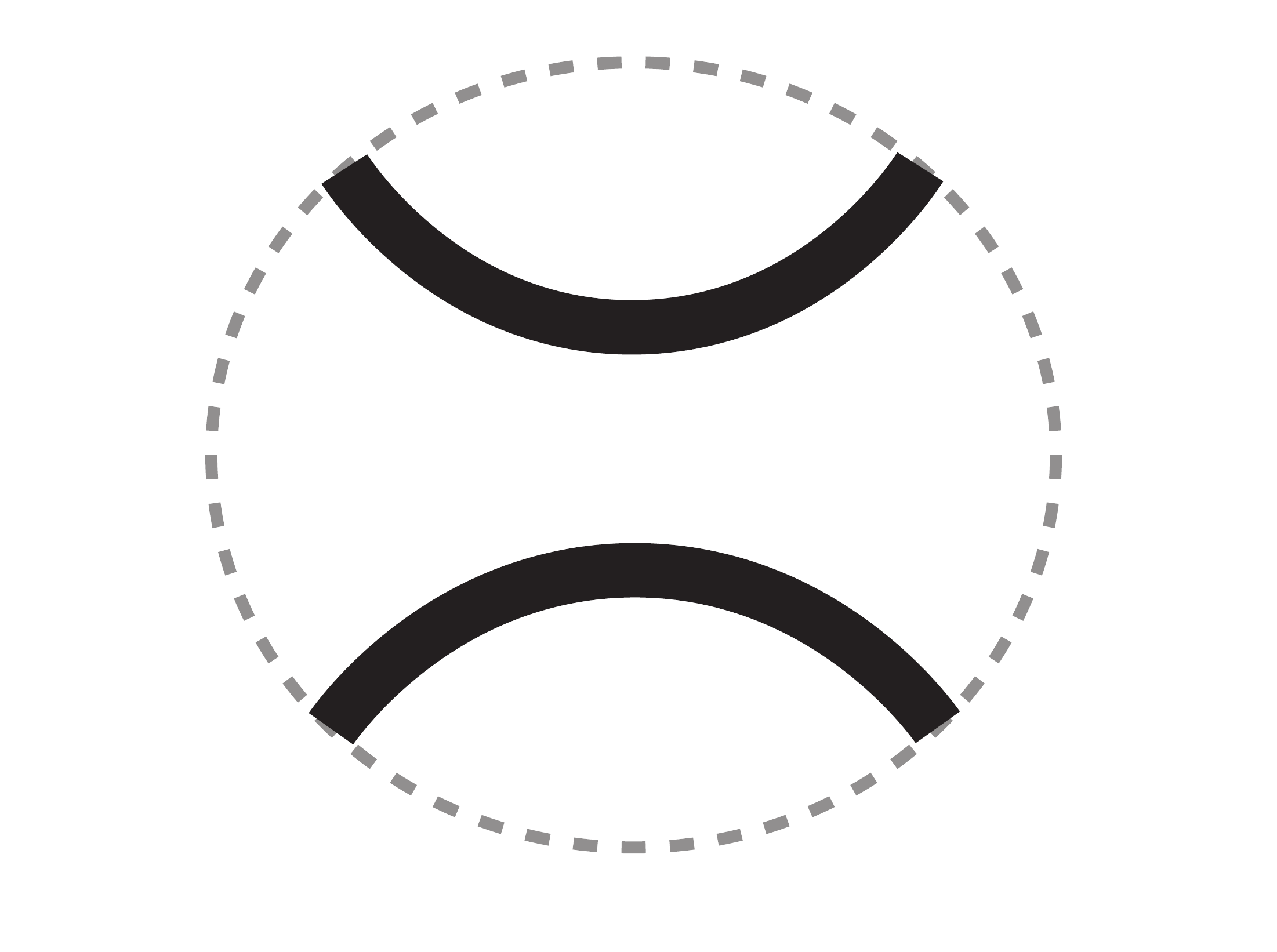}}} - q^{-1} \vcenter{\hbox{\includegraphics[scale=.07]{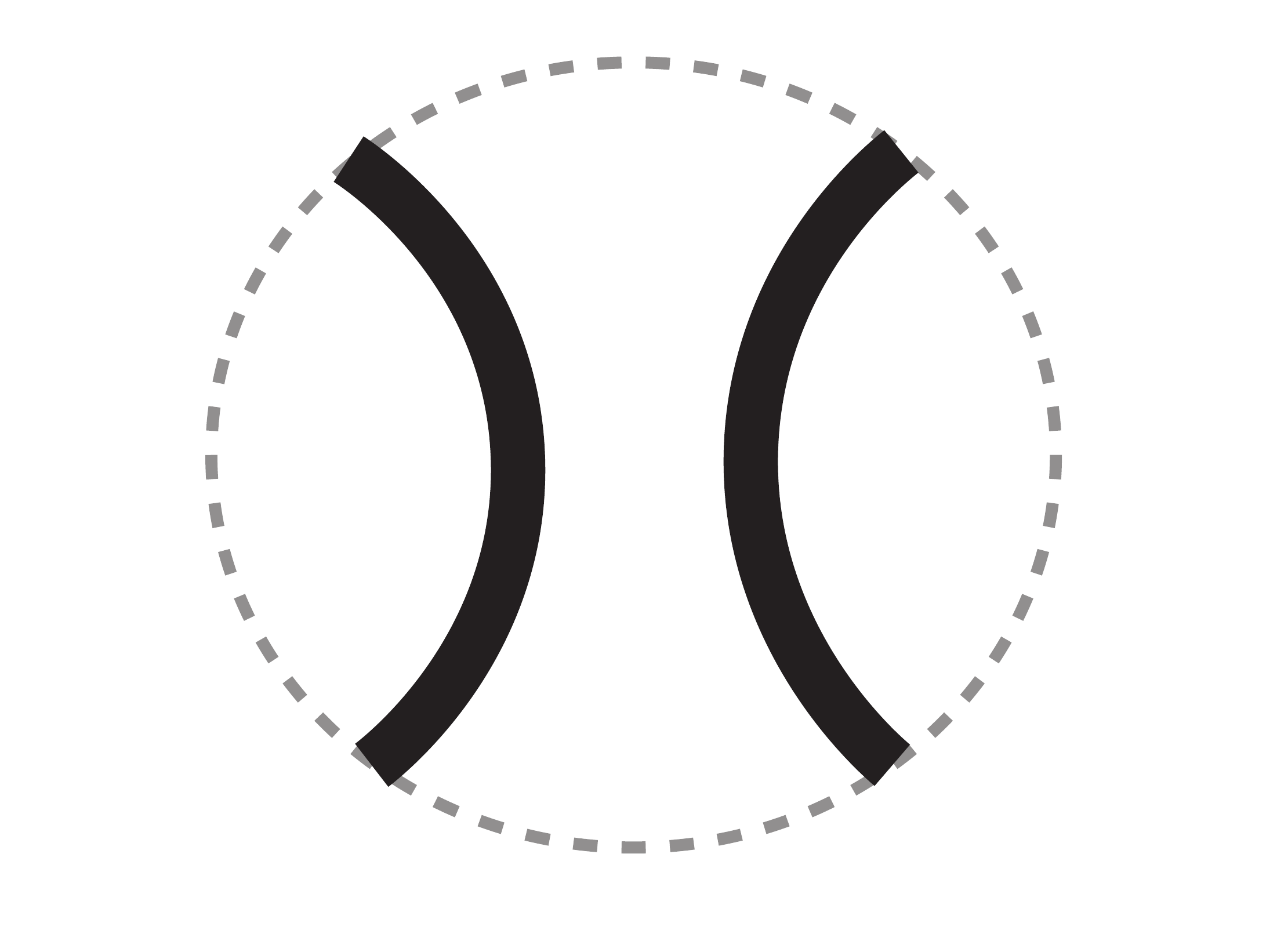}}} \ \text{and}\\
\label{trivialcomp} \vcenter{\hbox{\includegraphics[scale=.07]{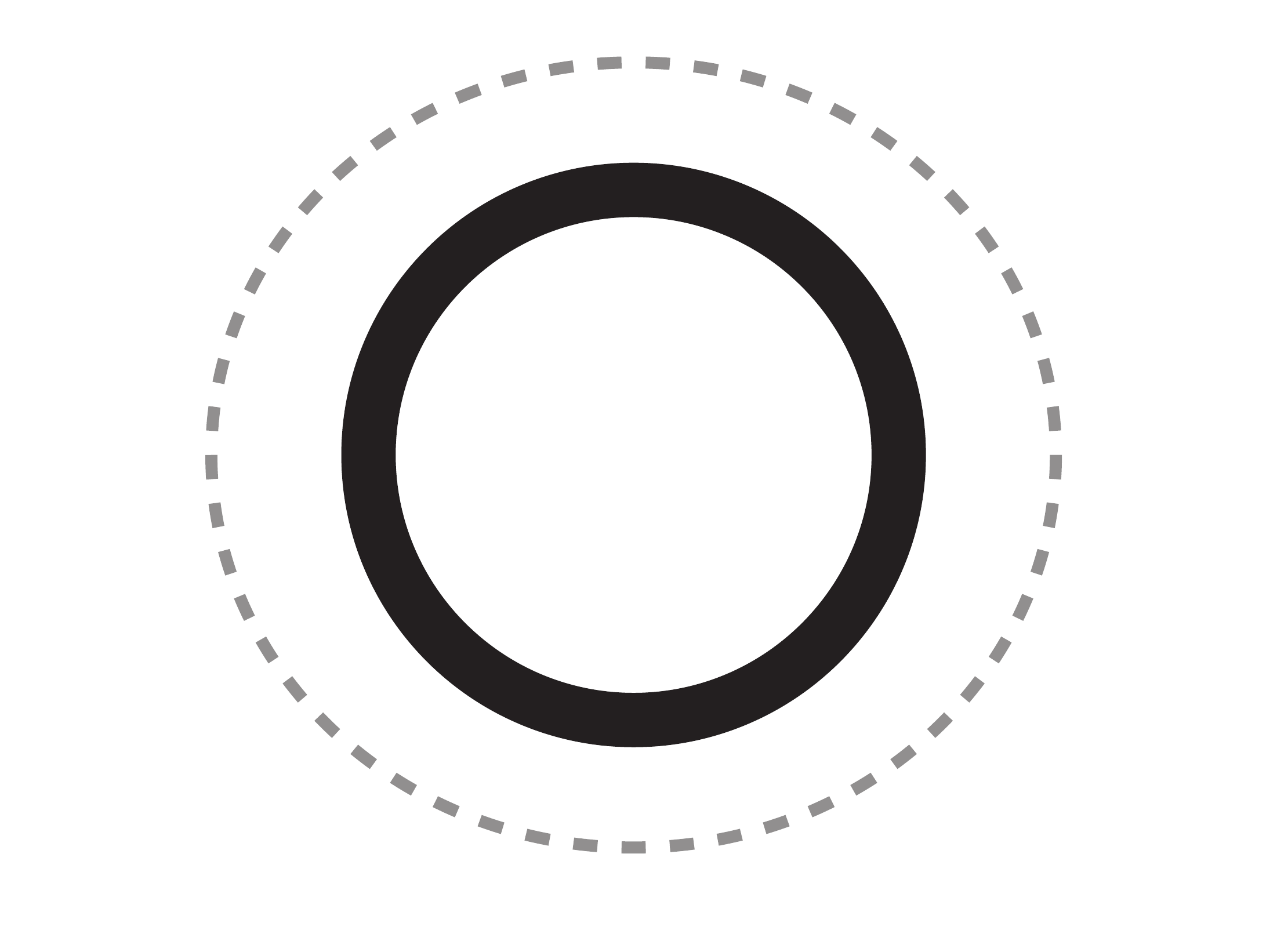}}} &+ q^2 + q^{-2}.
 \end{align}
Here the framed links in each expression are identical outside the $3$-balls pictured in the diagrams. For an unoriented framed link $L$ in $M$, we call the {\em skein class} of $L$ the element of $\SM$ determined by $L$. When there is no confusion, we will identify a framed link in $M$ with its skein class. An orientation preserving embedding $f: M \embed M'$ of oriented $3$-manifolds induces an $R$-linear map $f_*:\SM \to \cS(M')$, sending the skein class of a link $L$ to the skein class of $f(L)$ \cite{smof3,fundamentals}.
\def\sS{\mathscr S}

For an oriented surface $F$ with possibly non-empty boundary, let $\SF = \cS(F\times [-1,1])$. The $R$-module $\SF$ has an algebra structure, formally introduced by Turaev \cite{Turaev1} (see also \cite{HK,smof3}), where the product of two links is defined by placing the first link above the second in the direction given by the interval $[-1,1]$. The empty link is the identity. There is a $\BZ$-linear map $\tau : \SF \to \SF$, called the reflection,  defined by
 \be 
 \tau (q) = q^{-1}, \tau ( x, t) = (x, -t) \ \text{for } \ (x,t)\in F \times [-1,1].
 \ee
 We have $\tau(xy) = \tau(y) \tau(x)$ and $\tau^2=\id$. 
From the defining relations in the Kauffman bracket skein  module, one can easily derive the following relation, known as the {\it framing relation}.
\begin{align}
 \label{eq.kink}
 -q^{-3}\vcenter{\hbox{\includegraphics[scale=.07]{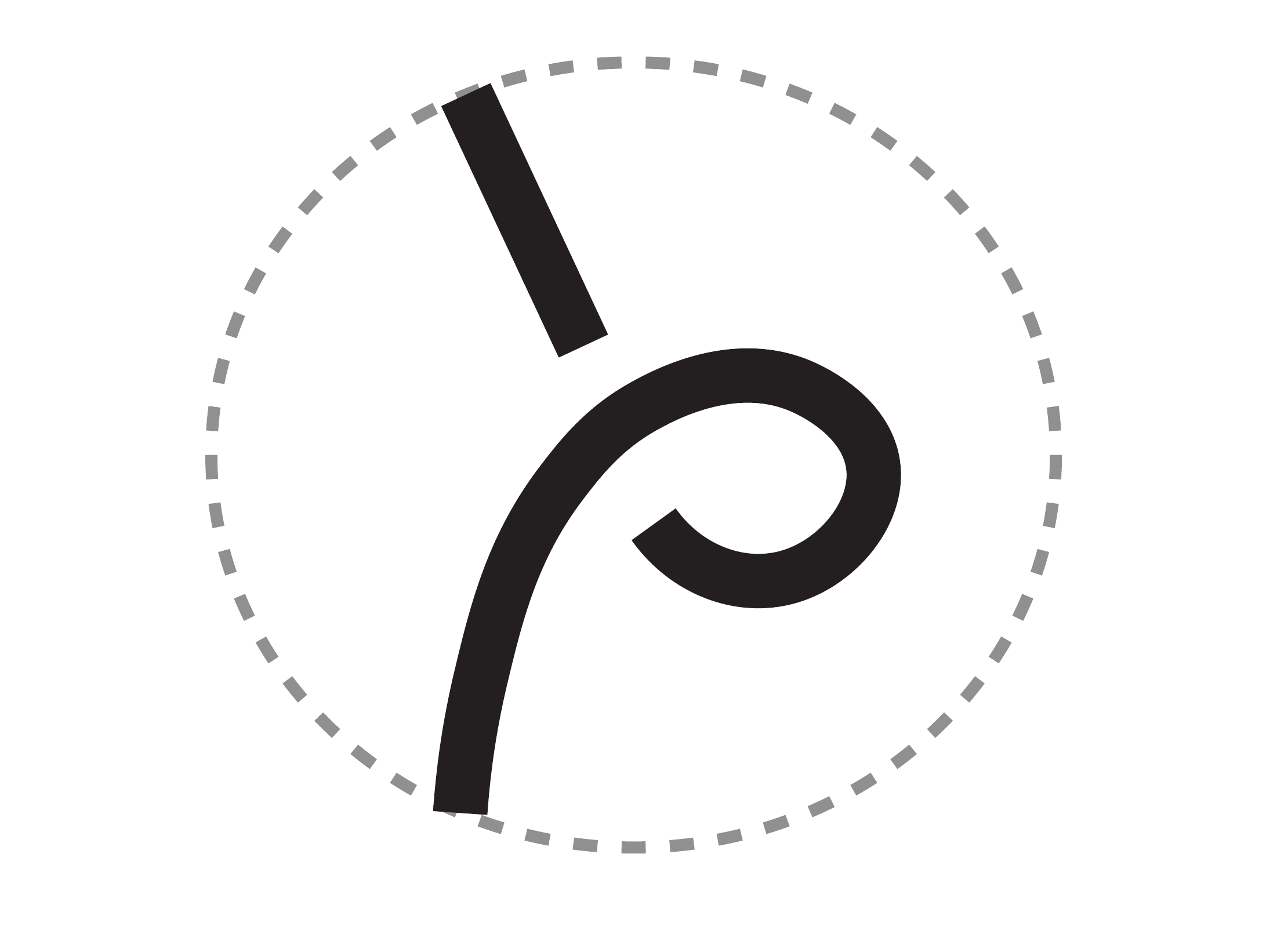}}} =\vcenter{\hbox{\includegraphics[scale=.07]{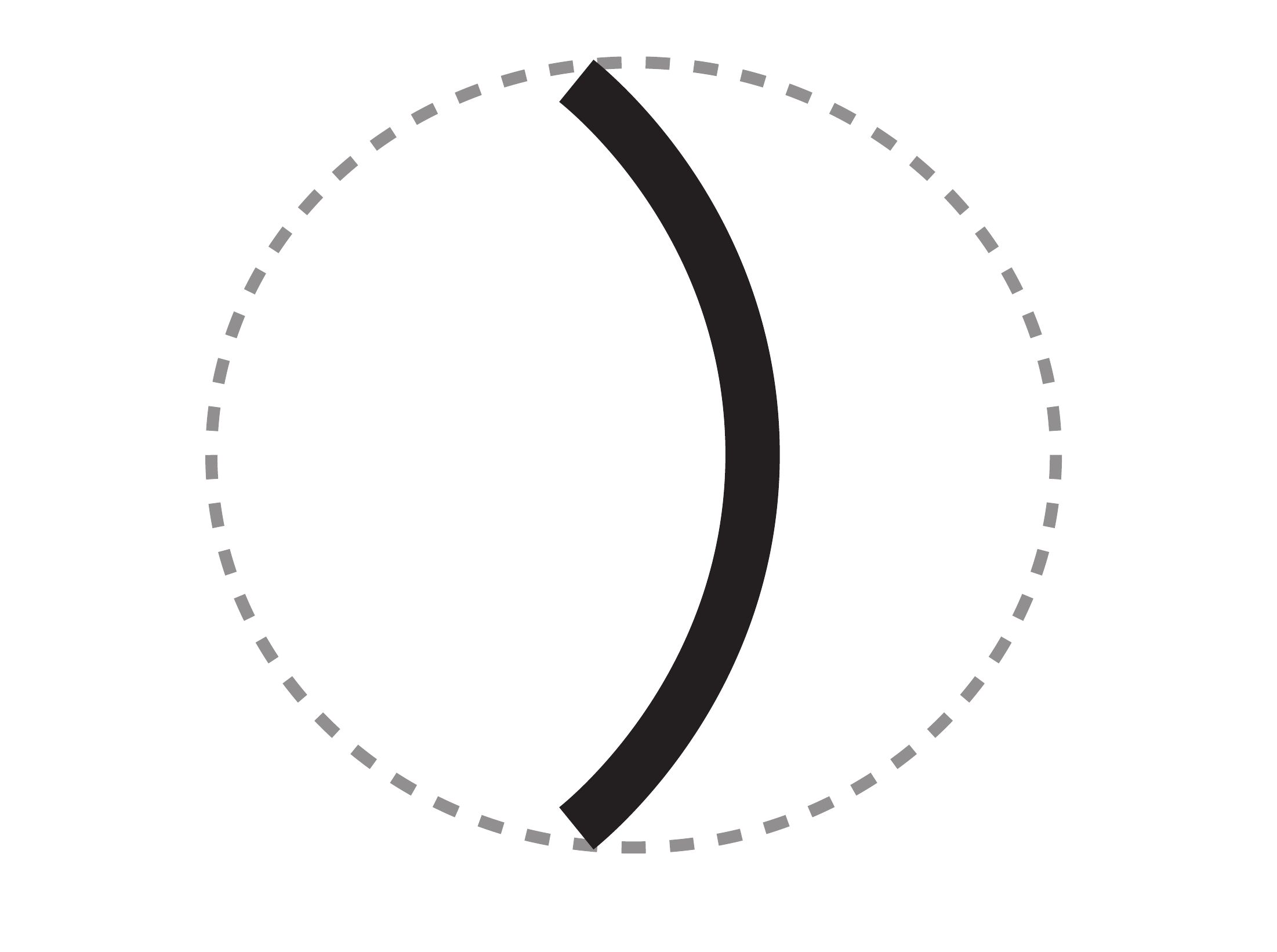}}} = -q^{3} \vcenter{\hbox{\includegraphics[scale=.07]{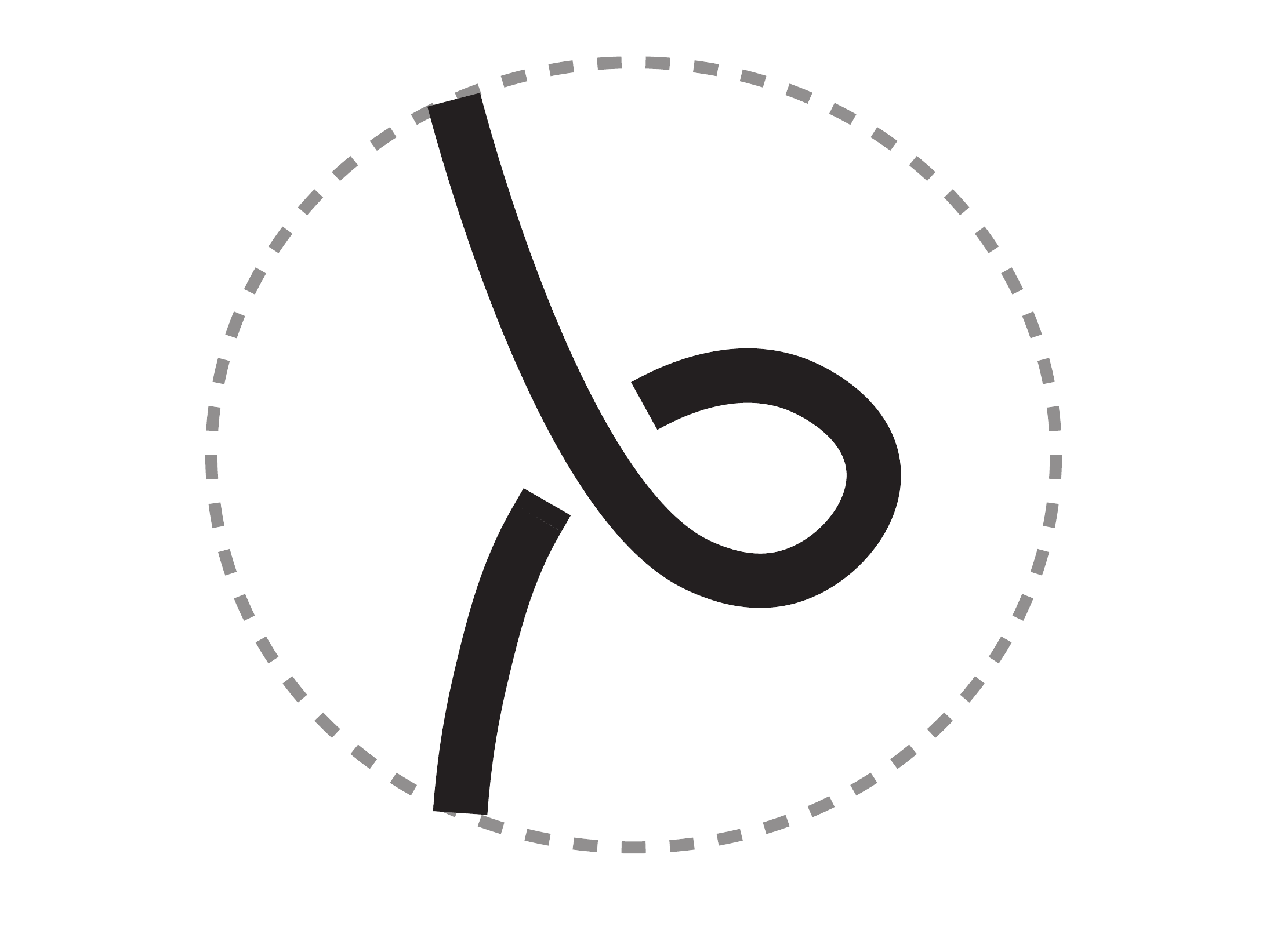}}}
\end{align}

\subsection{Attaching handles to $3$-manifolds} We now review how the Kauffman bracket skein module changes under the operation of $2$-handle attachment.

Let $M'$ be the result of attaching a $2$-handle to $M$. By \cite{smof3, fundamentals},  the embedding $M \embed M'$ induces a surjective $R$-homomorphism $ \SM 
\twoheadrightarrow 
\cS(M')$ whose kernel is spanned by $L - L'$, where the link $L$ is a generator of $\SM$ and $L'$ is a slide of $L$ over the $2$-handle. For a given framed link $L$ there are many slides of it. It is much more convenient to work with relative links, or tangles, where each relative link gives rise to exactly one sliding pair.

A {\em marked point} is a point $v\in \pM$ equipped with a ray in the tangent space of $\pM$ at $v$, starting at $v$. Suppose $u,v$ are two marked points in $\pM$. A {\em $(u,v)$-relative framed link} in $M$ is the disjoint union of a framed link in $M$ and a framed arc $\alpha$ whose endpoints are $u$ and $v$, such that the tangents of $a$ at $u$ and $v$ match the marking of $u$ and $v$ and the framing of $\alpha$ at $u$ and $v$ is the normal vectors of $M$. Two $(u,v)$-relative framed links are {\em isotopic} if there is an ambient isotopy of $M$ pointwise fixing $\pM$, which moves the first to the second.
The relative skein module $\cS(M;\{ u,v\})$, \cite{smof3,HP,fundamentals},  is the $R$-module freely spanned by isotopy classes of $(u,v)$-relative links modulo the same relations \eqref{kbsr} and \eqref{trivialcomp}.

Assume $M'$ is obtained by attaching a 2-handle to $M$ along a simple closed curve $\gamma$ on $\pM$. Choose two points $u,v\in \gamma$, and equip them with markings, which are locally on the same side of $\gamma$ in $\pM$. For a $(u,v)$-relative framed link $\al$, define the framed link $\al_{(1)}$ and $\al_{(2)}$ as in Figure \ref{bullofig}. Thus, $\al_{(i)}= \al \cup \gamma_i$ with the obvious smoothing at $u$ and $v$.

\begin{figure}[h]
    \centering
    \begin{subfigure}{.325\textwidth}
    \centering
\begin{overpic}[unit=1mm, scale = 0.1]{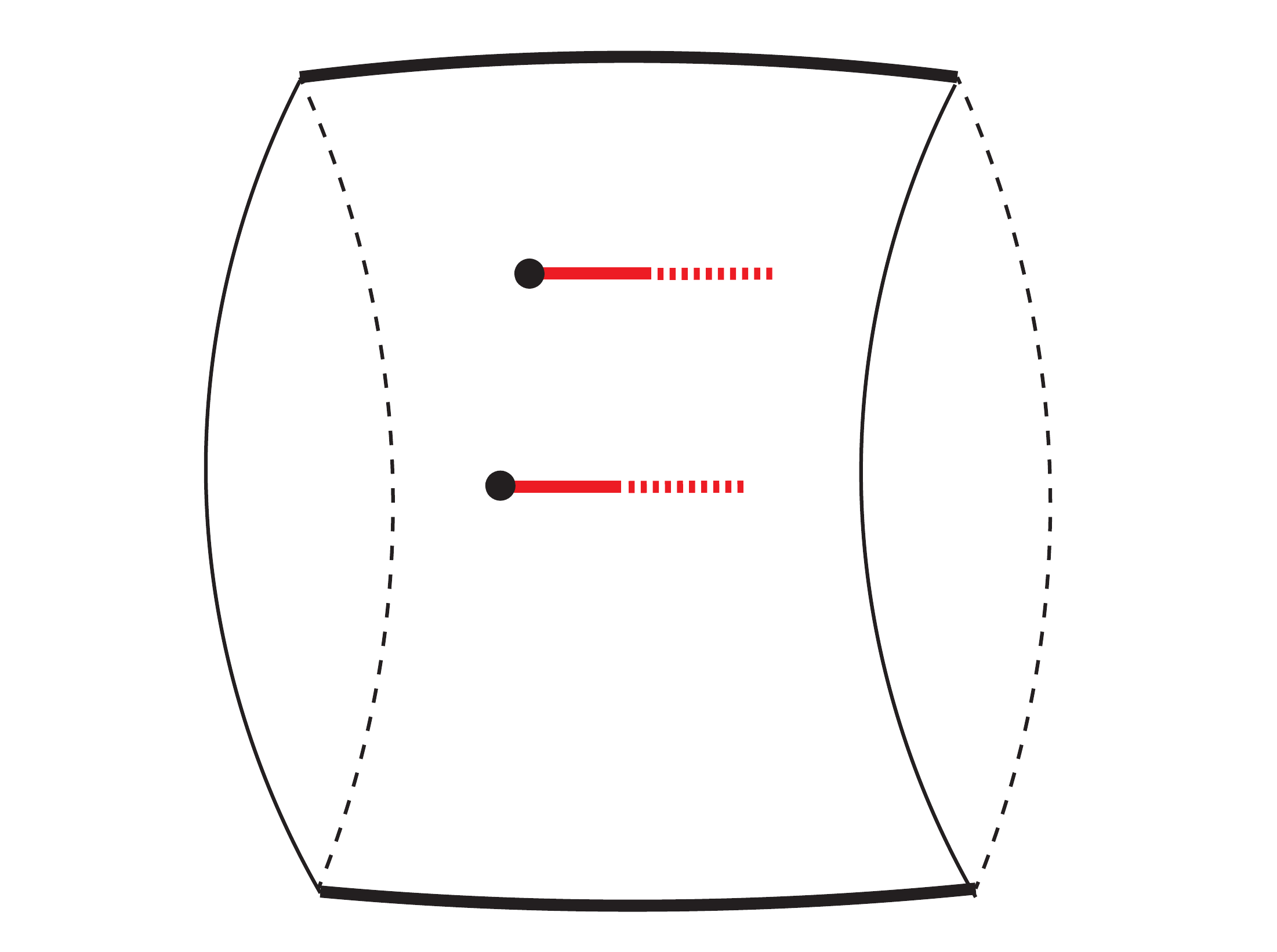}
 \put(12,13){\footnotesize{$v$}} 
 \put(13,19){\footnotesize{$u$}}
 \put(3,15){$\bigg\{$}
 \put(1.5,15.5){\tiny$\alpha$}
\end{overpic}
\subcaption{$\alpha \in \cS(M;\{u,v\})$}
 \label{bullofig1}
\end{subfigure}
\begin{subfigure}{.325\textwidth}
\centering
\begin{overpic}[unit=1mm, scale = .1]{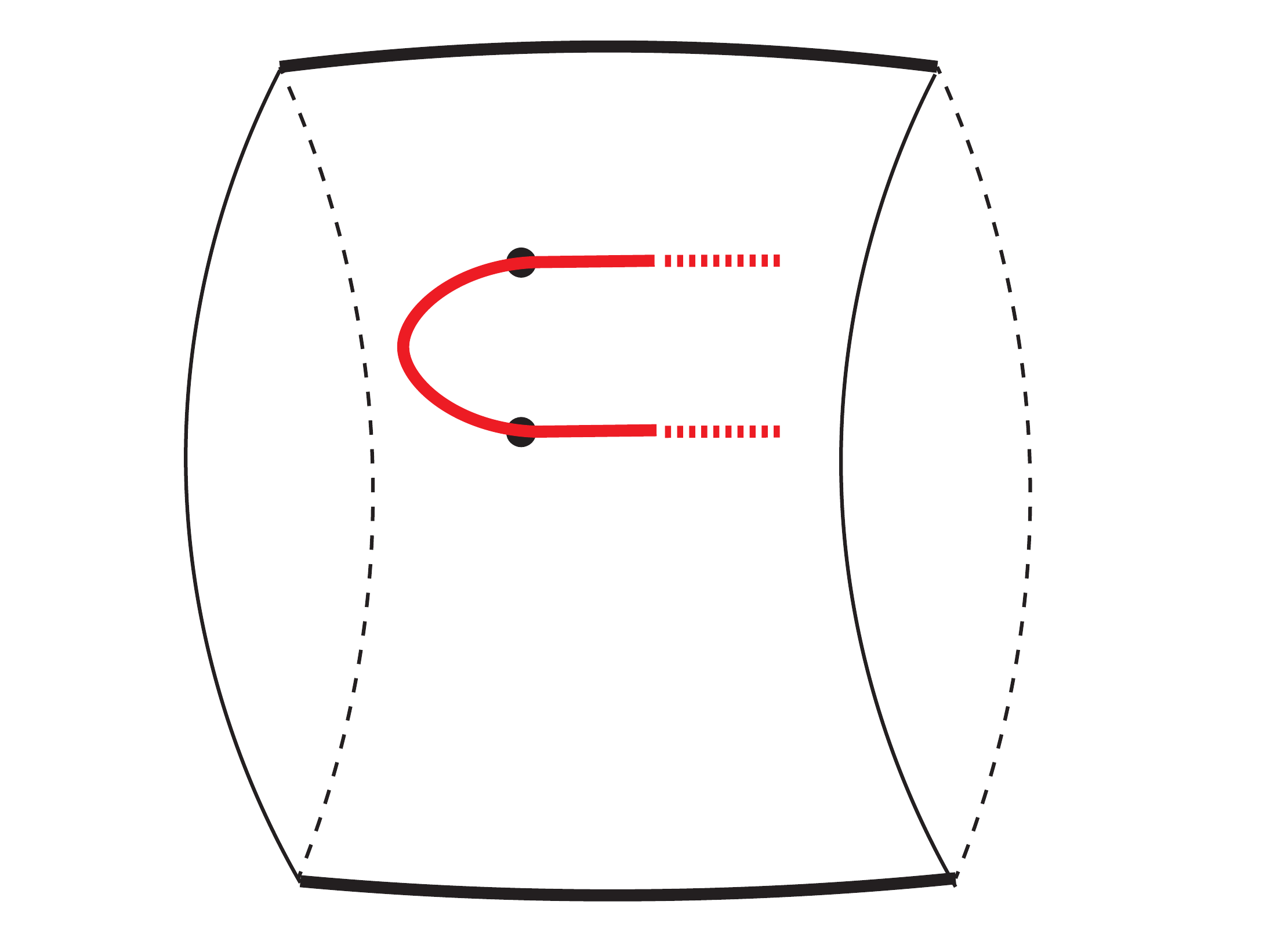}
 \put(10.5,21.5){\footnotesize{ $\gamma_2$}}

\end{overpic}
\subcaption{$\alpha_{(2)}=\alpha \cup \gamma_2$}
\end{subfigure}
\begin{subfigure}{.325\textwidth}
\centering
\begin{overpic}[unit=1mm, scale = .1]{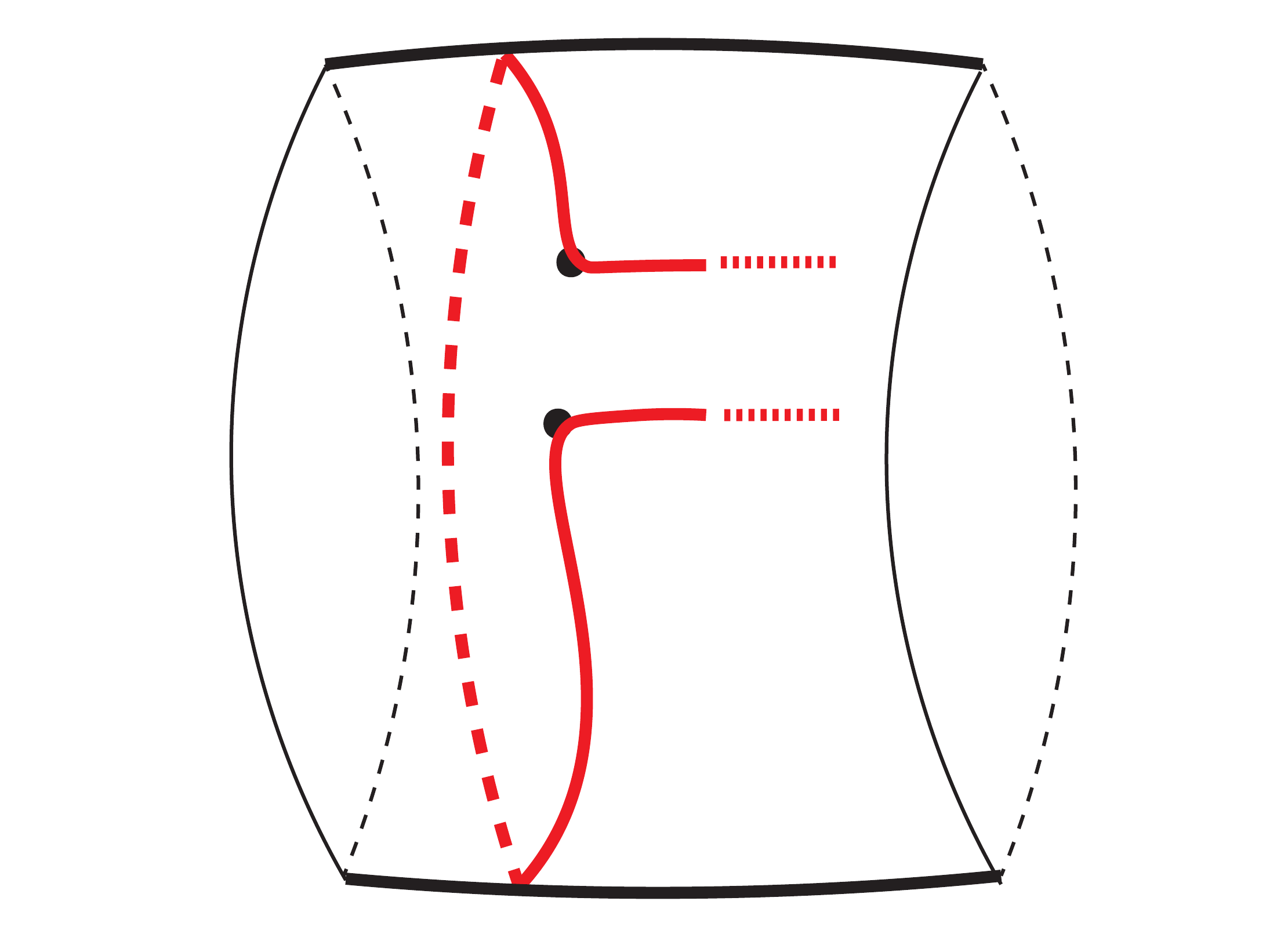}
 \put(13,17){\footnotesize{$\gamma_1$}}

\end{overpic} 
\subcaption{$\alpha_{(1)}=\alpha \cup \gamma_1$}
\end{subfigure}
    \caption{{\color{white}.}}
    \label{bullofig}
\end{figure}

A description of the kernel of this surjective $R$-homomorphism is the following.

\blem \label{r.slide}

The kernel of the surjective map $\cS(M) \onto \cS(M')$ is spanned by
$$ \{ w(\al):=\al_{(1)} - \al_{(2)} \ | \ \al \in \cS(M, \{u,v\})\}.$$
\elem


\section{Proof of main theorem}\label{sec.maintheorem}

\subsection{Setting and formulation}\label{ss.strategy} Let $H_d$ be a handlebody of genus $d$. It is easy to see that $\HH$ is the result of attaching a 2-handle to a genus two handlebody $H_2$ along the curve $\gamma$, as in Figure \ref{h1h1}.

\begin{figure}[ht]
    \centering
  $\vcenter{\hbox{\begin{overpic}[unit=1mm, scale = .6]{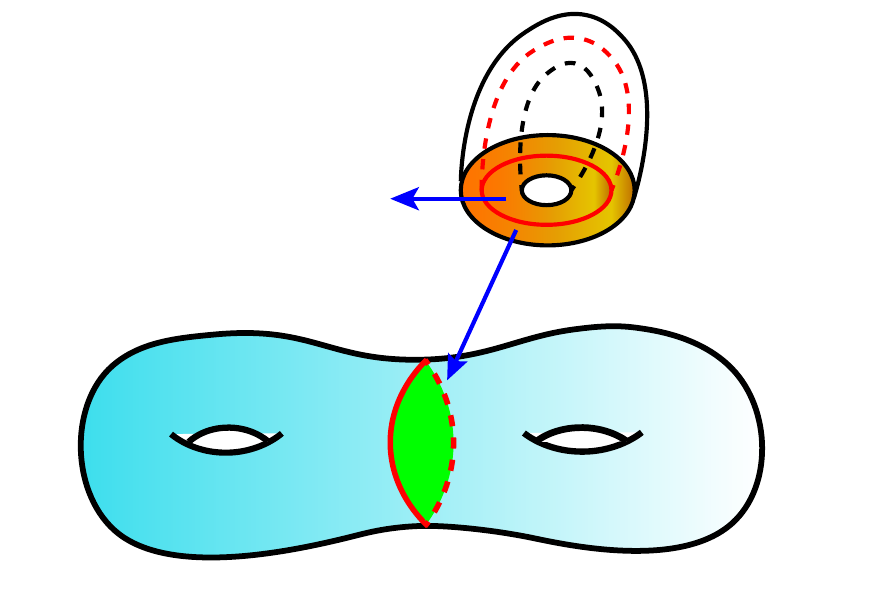}
 \put(37.5,19){$\gamma$} 
 \put(41.5,14){$D$} 
  \put(50,28){\footnotesize{\text{glue}}}
  \put(67,43){\footnotesize{$2$\text{-handle}}}
  \put(-7,13){$M'=$} 
  \put(17,28){$H_2 \cong D_2 \times I$}
  \put(35,39){$D_{\gamma}$}
  
\end{overpic} }}  $
    \caption{$H_1 \ \# \ H_1 $ is obtained by attaching a $2$-handle to $\partial H_2$ along the curve $\gamma$.}
    \label{h1h1}
\end{figure}
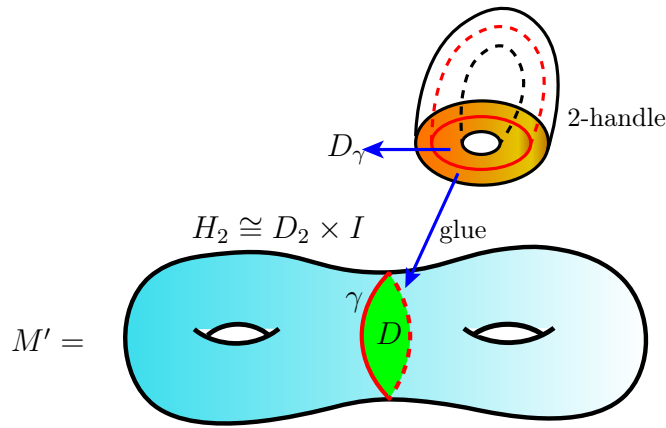 

\def\tD{\tilde D_2}

Let $D_2$ be the standard disk with two small open disks removed from its interior. Let $\tD$ be the thickened surface $\tD= D_2 \times [-1,1]$. Then $\tD \cong H_2$. Choose such a homeomorphism and identify $H_2$ with $\tD$. Then $\cS(H_2)= \cS(D_2)$ has an algebra structure. Let $x_1, x_2$, and $y$ be the curves on $D_2$ illustrated in Figure \ref{pantscurves}. By \cite{smquant}
\be 
\cS(D_2) = R[x_1, x_2, y]= \hR[y], \quad \text{where} \ \hR:=R[x_1, x_2].
\ee

\begin{figure}[ht]
\centering
\begin{overpic}[unit=1mm, scale = 0.12]{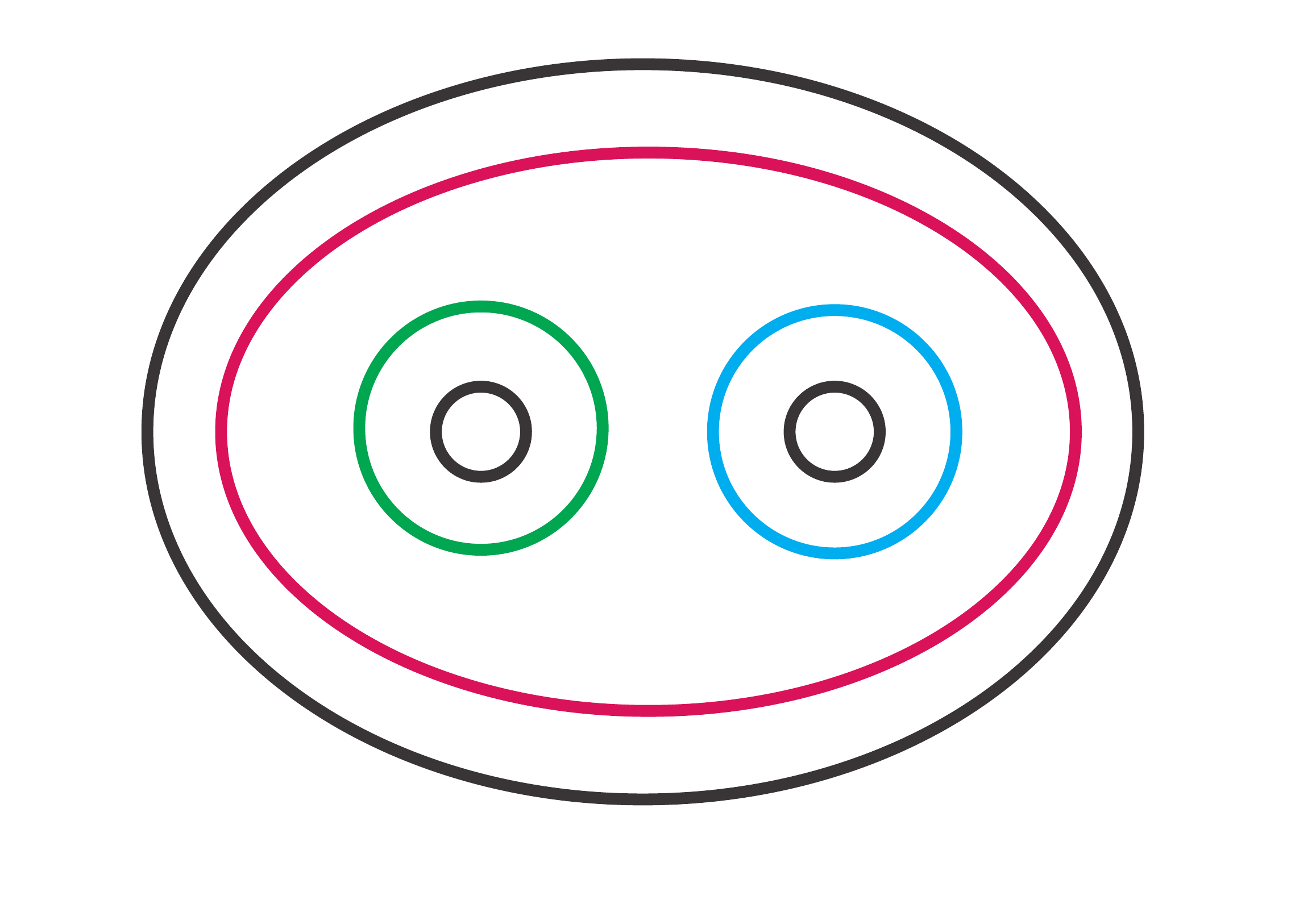}

\put(22,8){\footnotesize$y$}
\put(9.5,16){\footnotesize$x_1$}
\put(34.5,16){\footnotesize$x_2$}
\end{overpic}
\caption{Boundary parallel curves $x_1$, $x_2$, and $y$ in $F_{0,3}$.}
\label{pantscurves}
\end{figure}

The embedding $H_2 \embed \HH$ induces a surjective $R$-linear epimorphism
$$ \Phi: \cS(H_2)= \hR[y] \onto \SH.$$

We reformulate Theorem \ref{thm1} as follows.
\bthm \label{thm1a}
The kernel of $\Phi$ is the $\hR$-module $\fG$ spanned by $\{ G_n, n \ge 1\}$, where
\be 
G_n= \{ n+1\} S_n(y) + (-1)^{n+1} \{1\} S_n(x_1) S_n(x_2).
\ee
Consequently, $\Phi$ induces an isomorphism
$$ \phi: R[x_1, x_2, y]/\fG \cong \SH.  $$
\ethm

\brem The kernel of $\Phi$ was suggested by the work \cite{CL}, where it is essentially proved that $\{n\} G_n$ is in the kernel. 
\erem

\subsection{Strategy of Proof} 

Since the curves $x_1$ and $x_2$ are far from the attaching curve $\gamma$, if $z\in \ker \Phi$ then $x_1 z, x_2 z \in \ker\Phi$. In other words, $\ker \Phi$ is an $\hR$-module. Hence, $\SH$ is also an $\hR$-module and $\Phi$ is $\hR$-linear.


\def\SDr{\cS(\tD, \{u, v\})}

Consider the relative skein module $\SDr$, where  
 the marked points $u$ and $v$ lie on
$\gamma \cap (D_2 \times \{1\})$ and have markings pointing to the right as illustrated in Figure \ref{rkbsmbasis3}. There is a right action of
$\cS(D_2)$ on $\SDr$, 
$$\SDr \ot \cS(D_2) \to \SDr, \ \al \ot \beta \to \al * \beta,    $$
where  $\al * \beta$ is the result of stacking $\al$ above $\beta$. The following is  essentially \cite[Lemma 5.2]{2bk}, which is 
 a reformulation of a result in \cite{knotext}.
 

 \begin{figure}[h]
    \centering
$\vcenter{\hbox{\begin{overpic}[scale=.075]{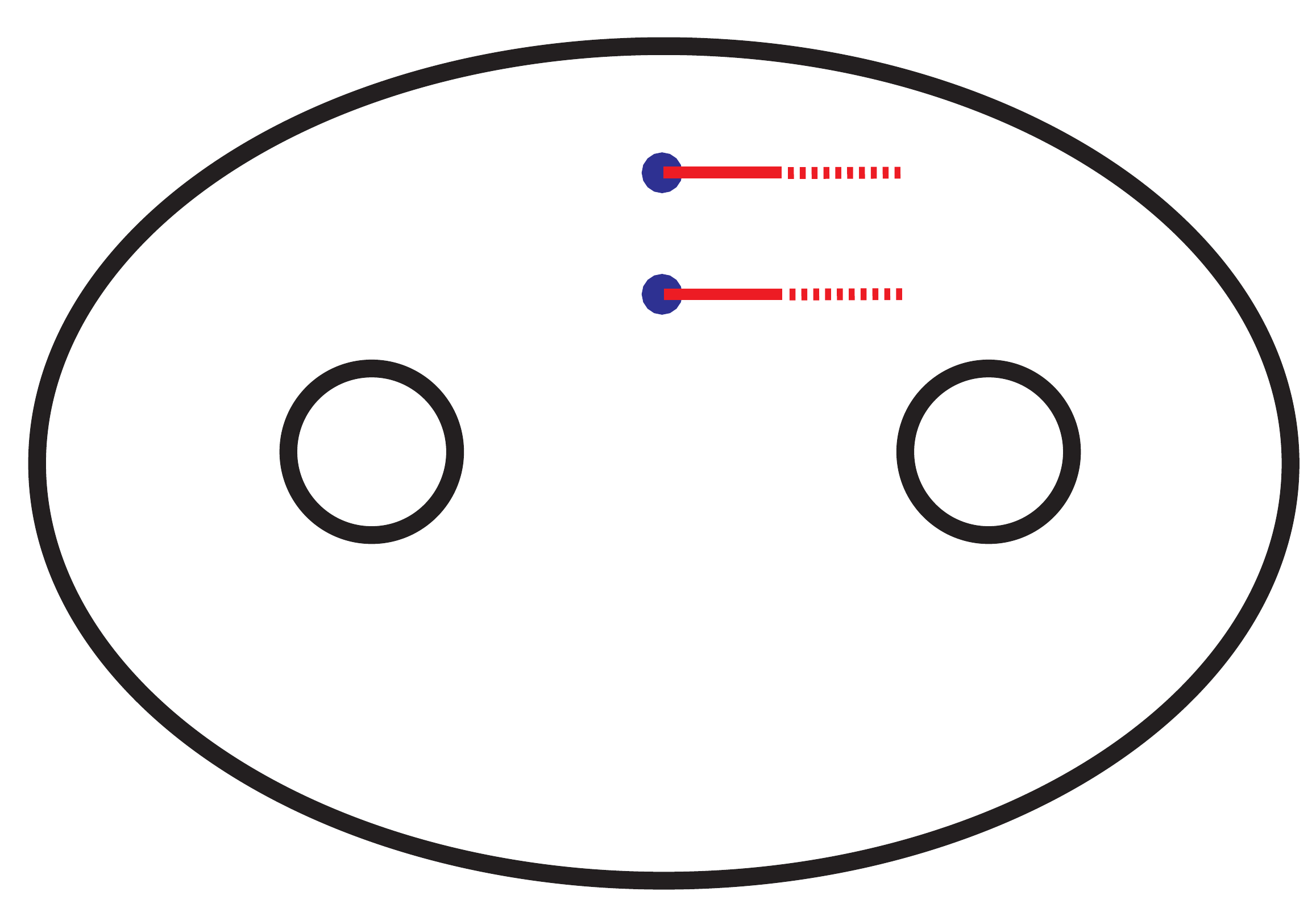}
\put(35, 45){\tiny{$u$}}
\put(35, 37){\tiny{$v$}}
\end{overpic}}}
\vcenter{\hbox{\begin{overpic}[scale=.075]{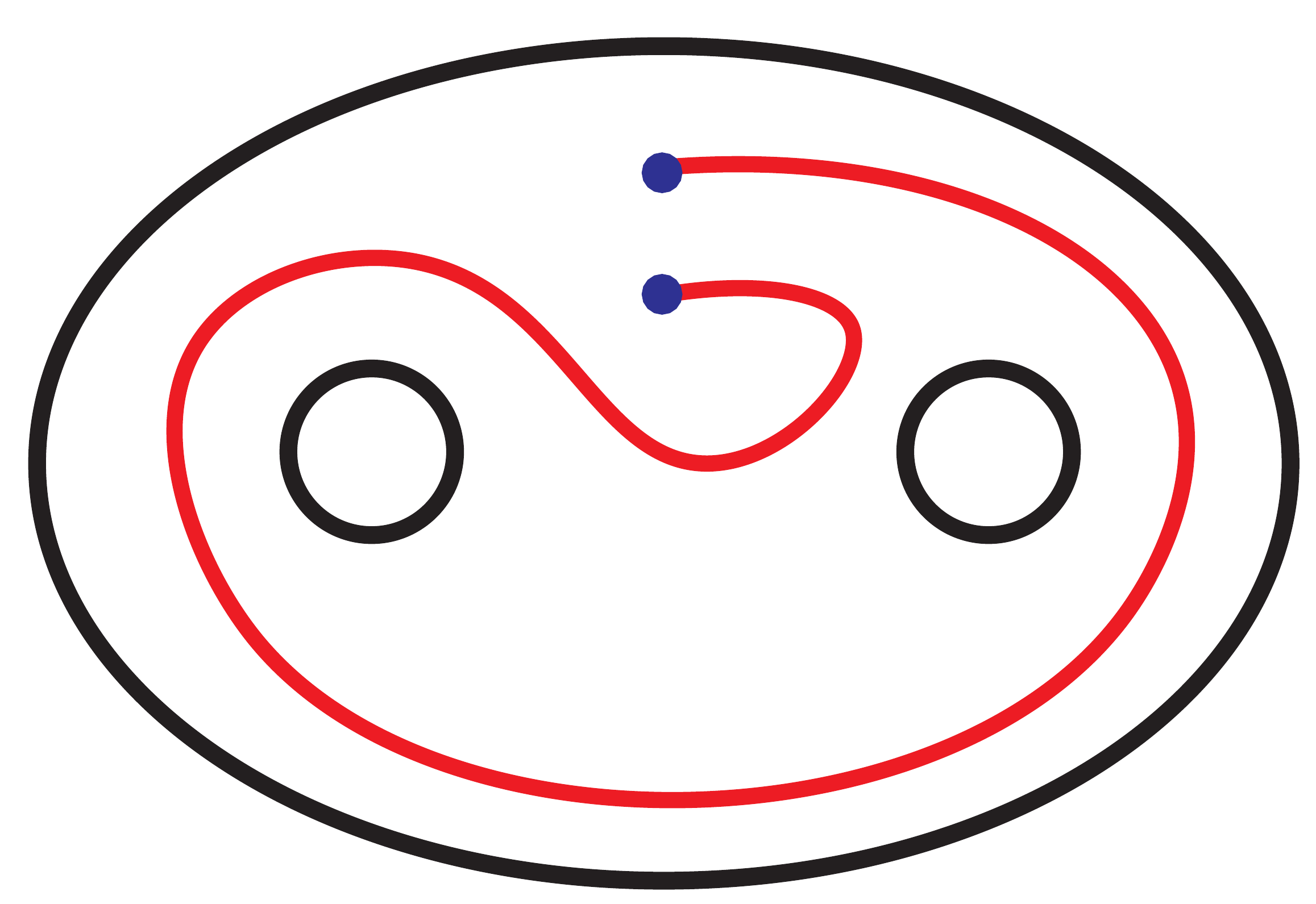}
\put(35, -7){$a_{1}$}
\end{overpic}}} \hspace{2mm}
\vcenter{\hbox{\begin{overpic}[scale=.075]{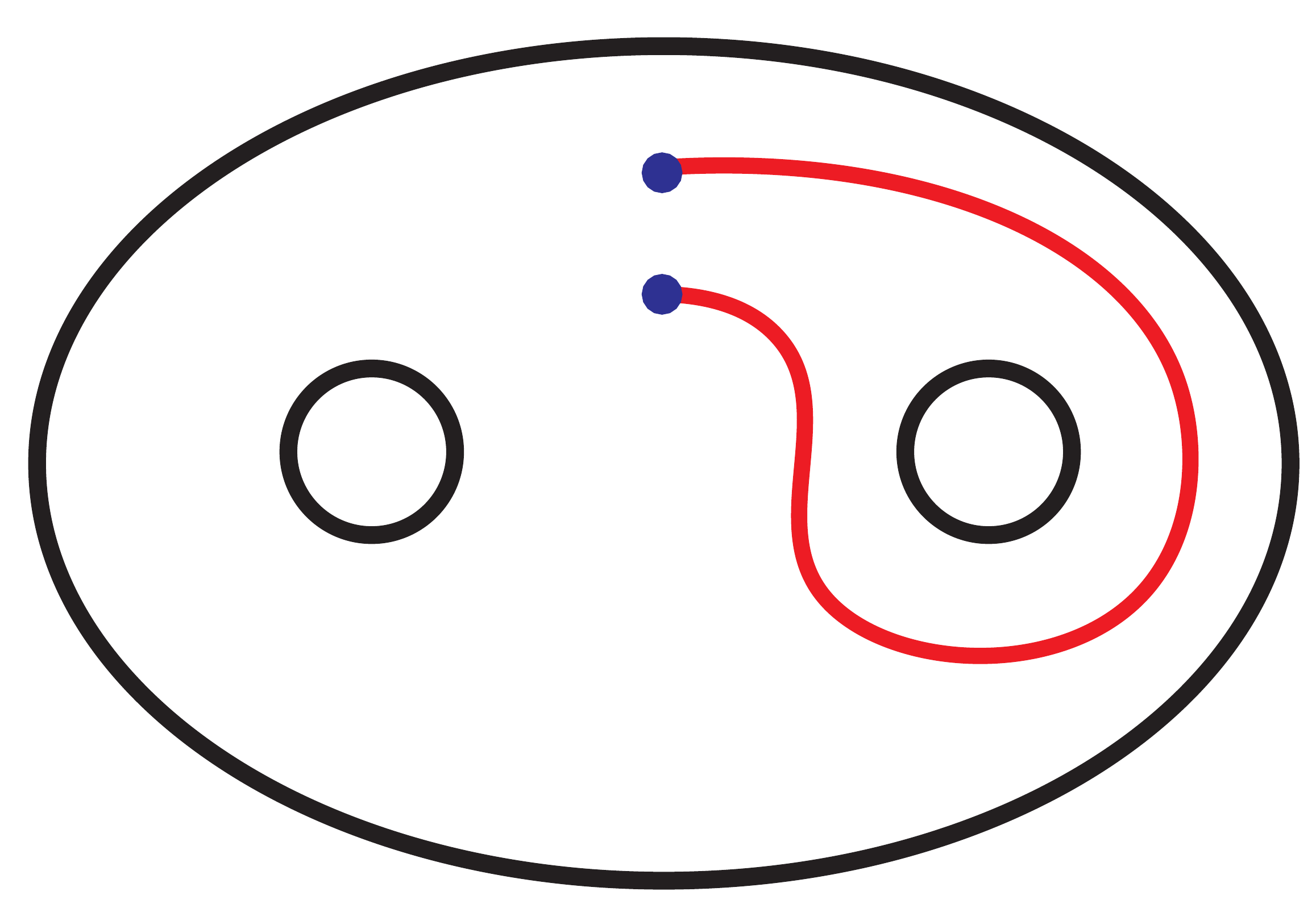}
\put(38, -7){$a_{2}$}
\end{overpic}}} \hspace{2mm}
\vcenter{\hbox{\begin{overpic}[scale=.075]{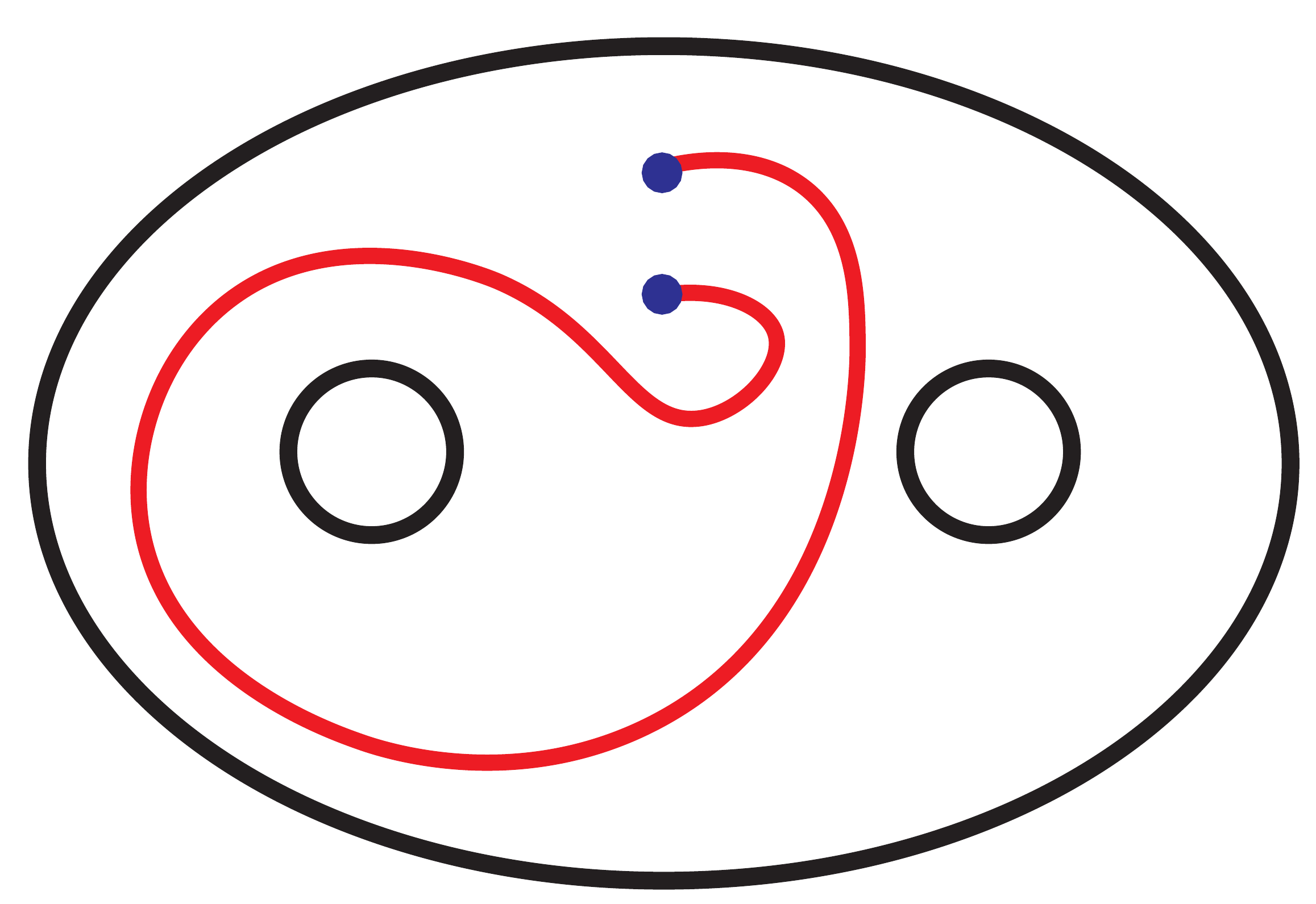}
\put(38, -7){$a_{3}$}
\end{overpic}}}  \hspace{2mm}
\vcenter{\hbox{\begin{overpic}[scale=.075]{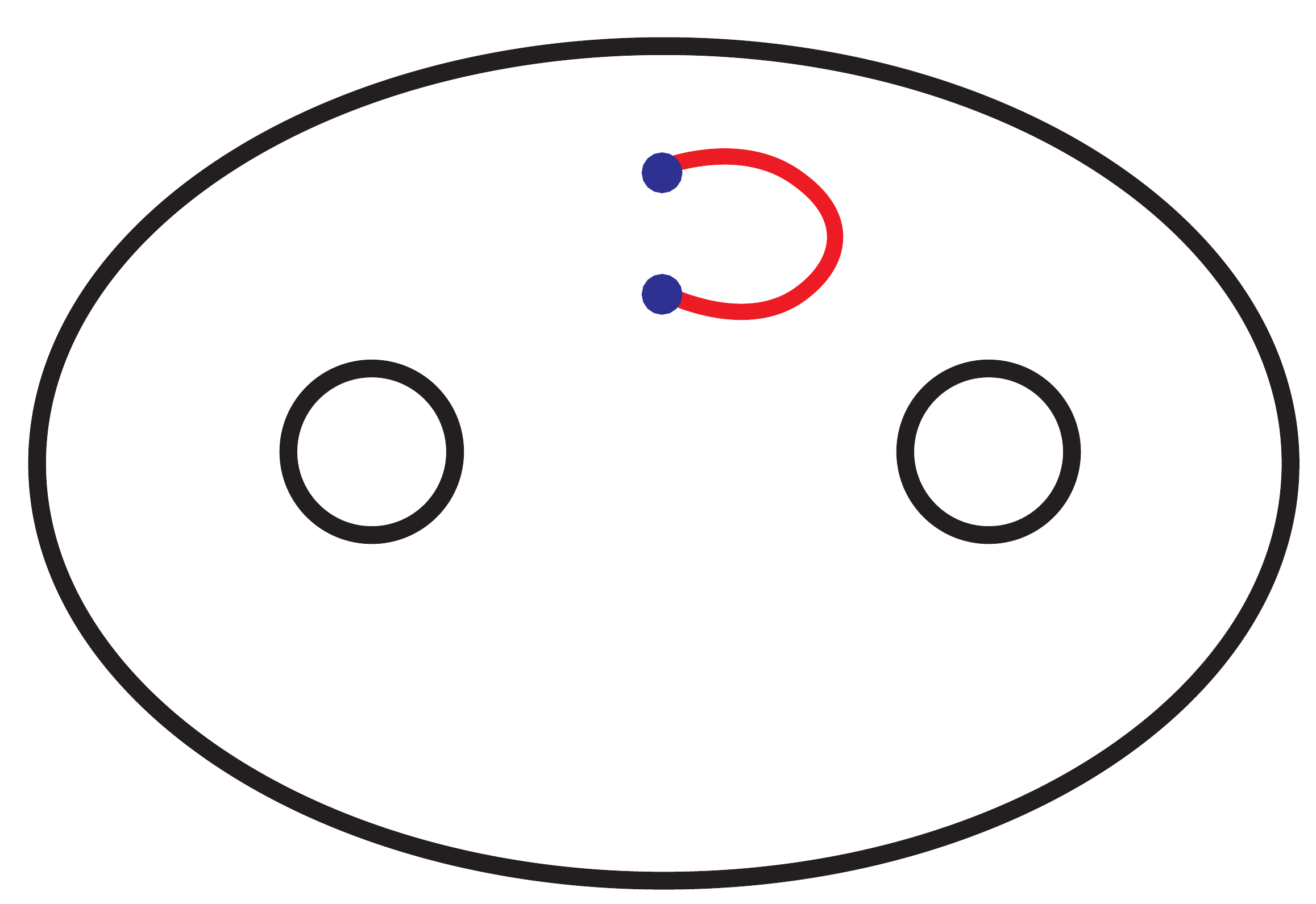}
\put(39, -7){$a_{4}$}
\end{overpic}}}$
\vspace*{2mm}
    \caption{The marked points $u$ and $v$ and the arcs $a_1, a_2, a_3$, and $a_4$.}
    \label{rkbsmbasis3}
\end{figure}

\blem As a right module over $\cS(D_2)$,  $\SDr$ is spanned by $a_1, a_2, a_3$, and $a_4$, which are illustrated in Figure \ref{rkbsmbasis3}.
\elem

Consequently, from Lemma \ref{r.slide} we have
\be 
\ker \Phi = \cK_1 + \cK_2 + \cK_3 + \cK_4,
\label{eq.K}
\ee
 where
\be \cK_i= w(a_i * \hat R [y]) =  \sum_{k=0}^\infty w (a_i* y^k) * \hR.
\ee

Let us have a closer look at $\cK_1$, which is $\hR$-spanned by $w(a_1 * y^k), k\ge 1$. We have
\begin{align}
(a_1 * y^{k-1})_{(1)} 
 & = \vcenter{\hbox{\begin{overpic}[scale=.075]{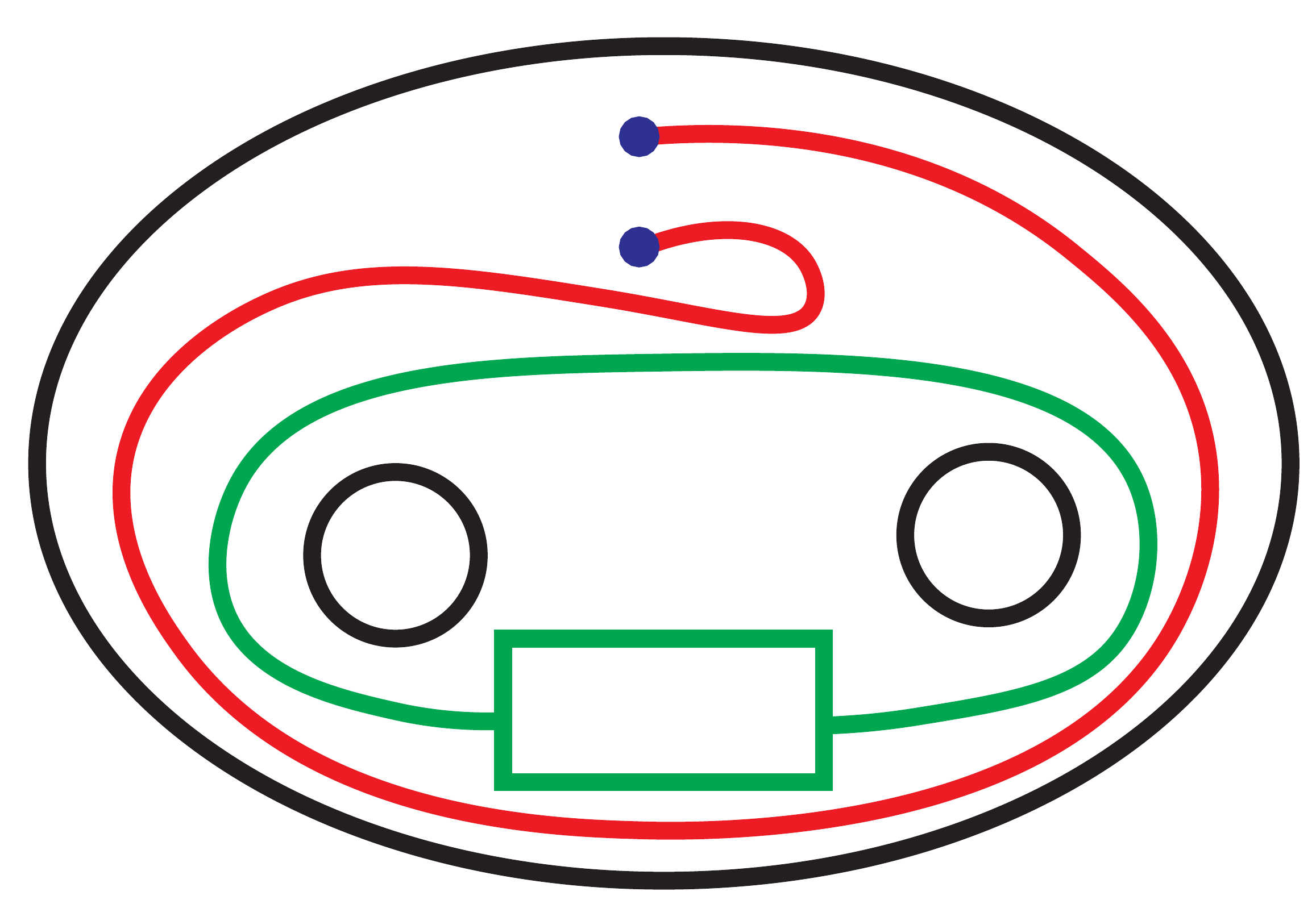}
 \put(33,9.5){\tiny{$k-1$}}
\end{overpic}}} = y^k \\
(a_1 * y^{k-1})_{(2)} &= \vcenter{\hbox{\begin{overpic}[scale=.075]{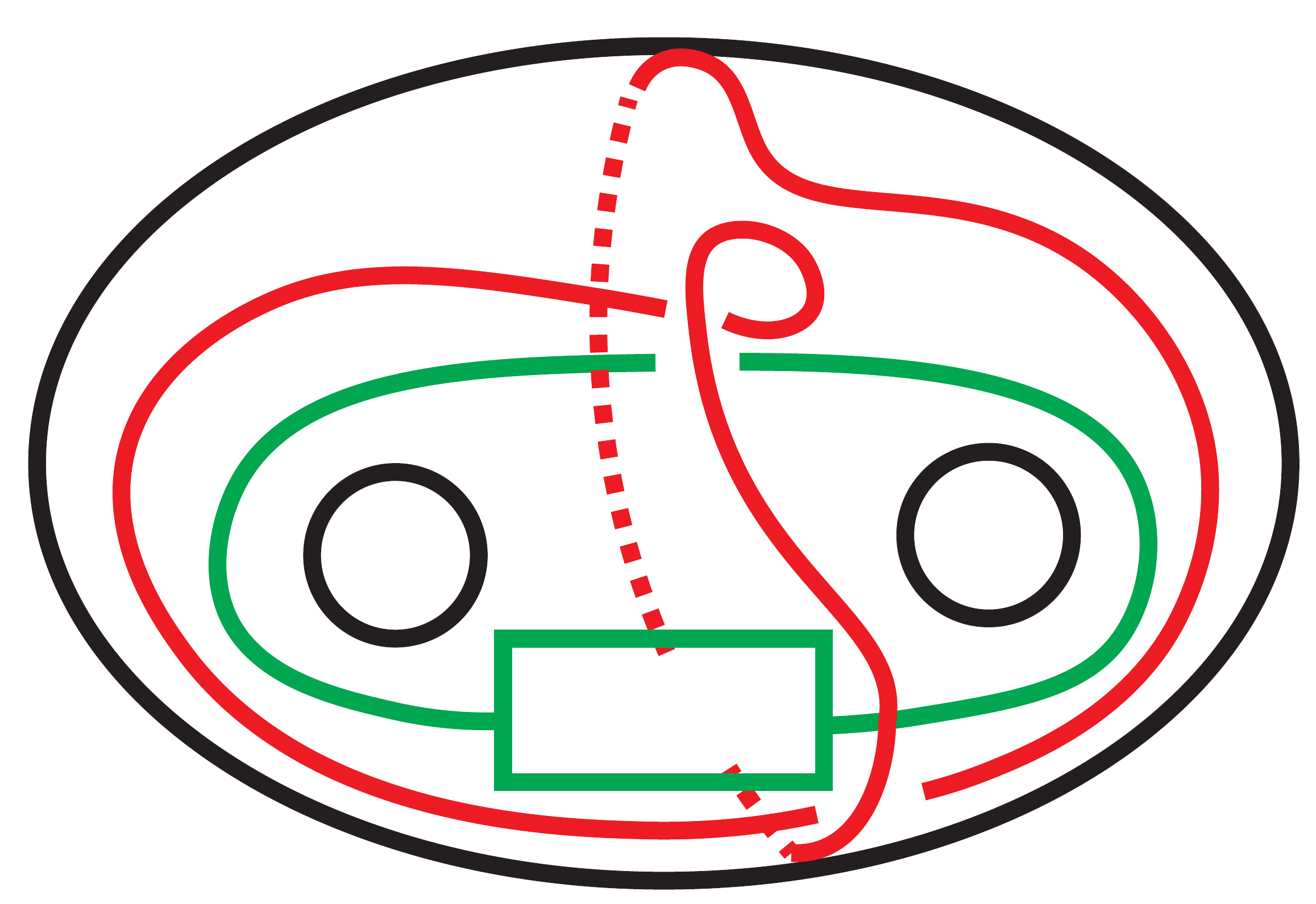}
\put(33,9.5){\tiny{$k-1$}}
\end{overpic}}} = q^6 z_k,
\end{align}
where $z_k \ \text{(due to Equation \ref{eq.kink})} = \vcenter{\hbox{\begin{overpic}[scale=.075]{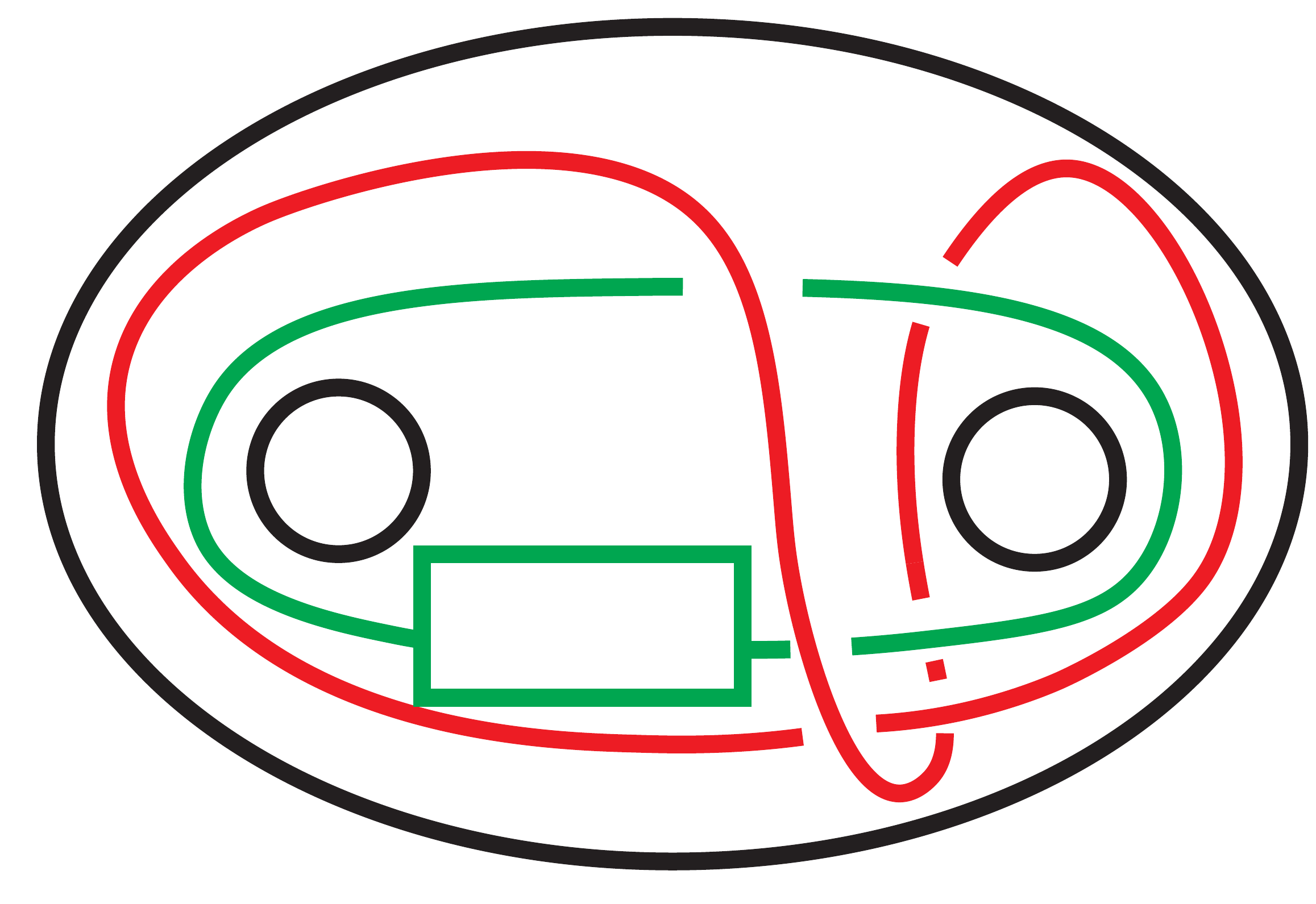}
 \put(27.5,15){\tiny{$k-1$}}
\end{overpic}}} .$ 
Thus, we have the following explicit description of $\cK_1$.
\blem The $\hR$-module
$\cK_1$ is  spanned by $\{ y^k - q^6 z_k, k \ge 1\}$.
\elem

The conjecture of the first and third authors in \cite{counterhandle} mentioned in Section \ref{introsectionh1h1} is the following statement.
\bthm We have $\ker \Phi=\cK_1$.
\ethm
 We will see that this  is a consequence of the proof of Theorem  \ref{thm1a}. We now prove the following lemmas, which prove Theorem \ref{thm1a}.

\blem \label{r1} One has $\cK_1= \fG$.
\elem

\blem \label{r234}
 For each $i=2,3,4$ we have $\cK_i \subset \cK_1$.
\elem

Lemma \ref{r234} and Identity \eqref{eq.K} show that $\ker \Phi= \cK_1$, which proves the conjecture by the first and third authors. Lemma \ref{r1} further shows that $\ker \Phi= \fG$, proving Theorem \ref{thm1a}. 

\subsection{Proof of Lemma \ref{r234} assuming Lemma \ref{r1}}
We need to prove that $w(a_i z)\in \cK_1$ for $i=2,3,4$ and $z=y^k, k =0,1,2,\dots$. 
\bpr[Proof that $\cK_2 \subset \cK_1$] By definition,  $w (a_2 z) = (a_2 z)_{(1)} - (a_2 z)_{(2)}$. The diagrams for these curves are illustrated in Figure \ref{fig: a2z}.

\begin{figure}[h]
    \centering
$\vcenter{\hbox{\begin{overpic}[scale=.075]{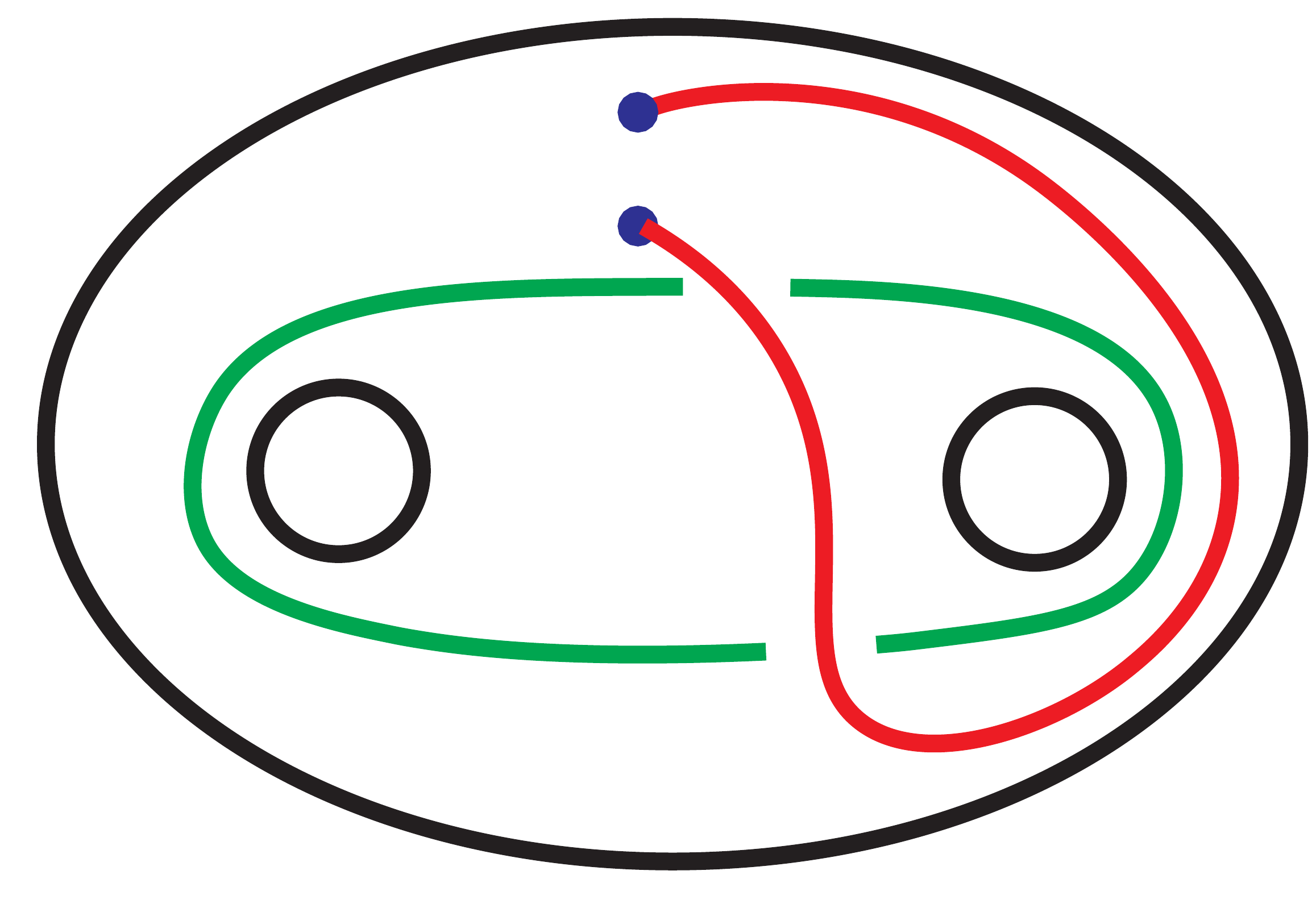}
\put(35, -8){$a_{2}z$}
\end{overpic}}} \hspace{4mm}
\vcenter{\hbox{\begin{overpic}[scale=.075]{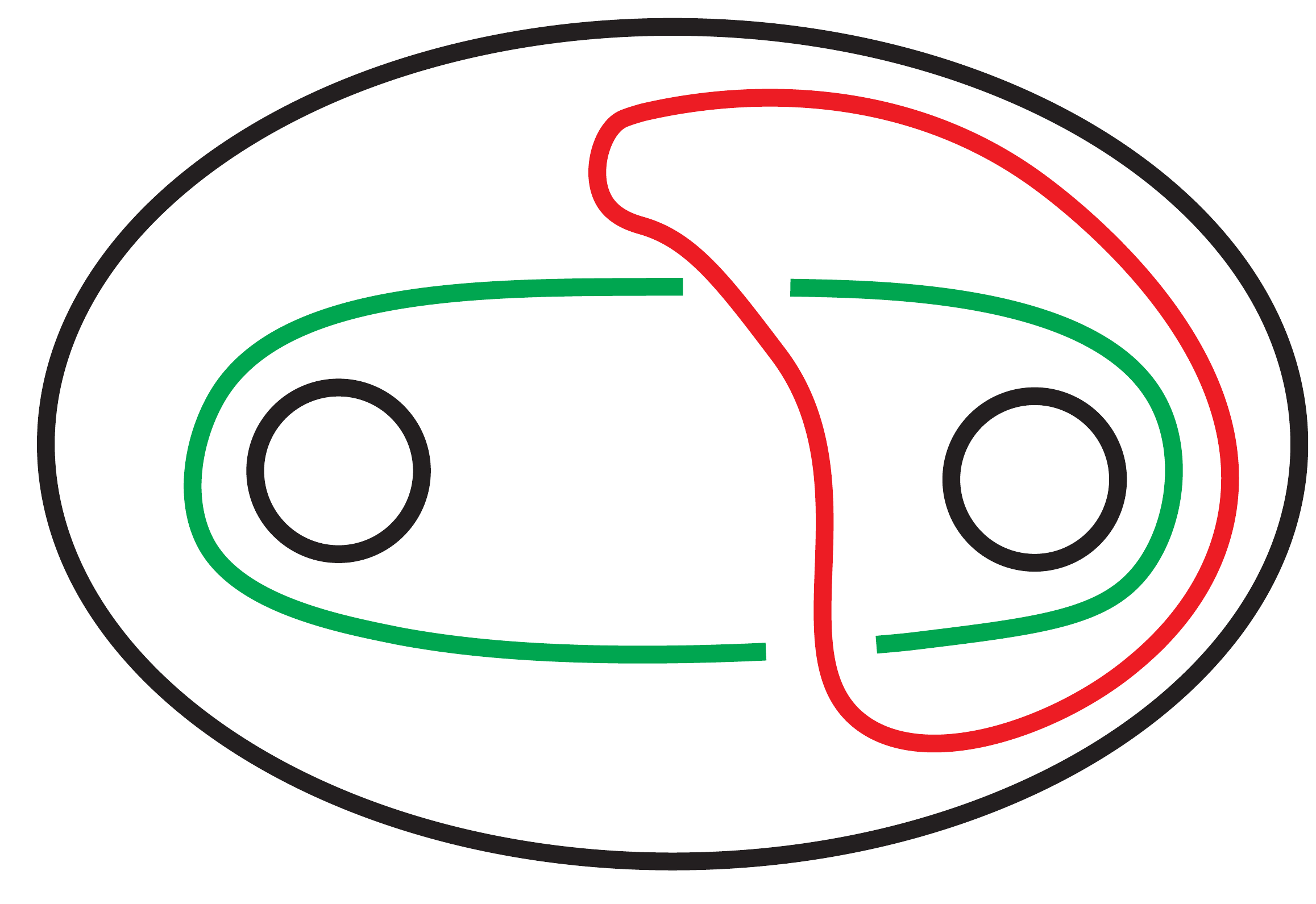}
\put(30, -9){$(a_{2}z)_{(1)}$}
\end{overpic}}} \hspace{4mm}
\vcenter{\hbox{\begin{overpic}[scale=.075]{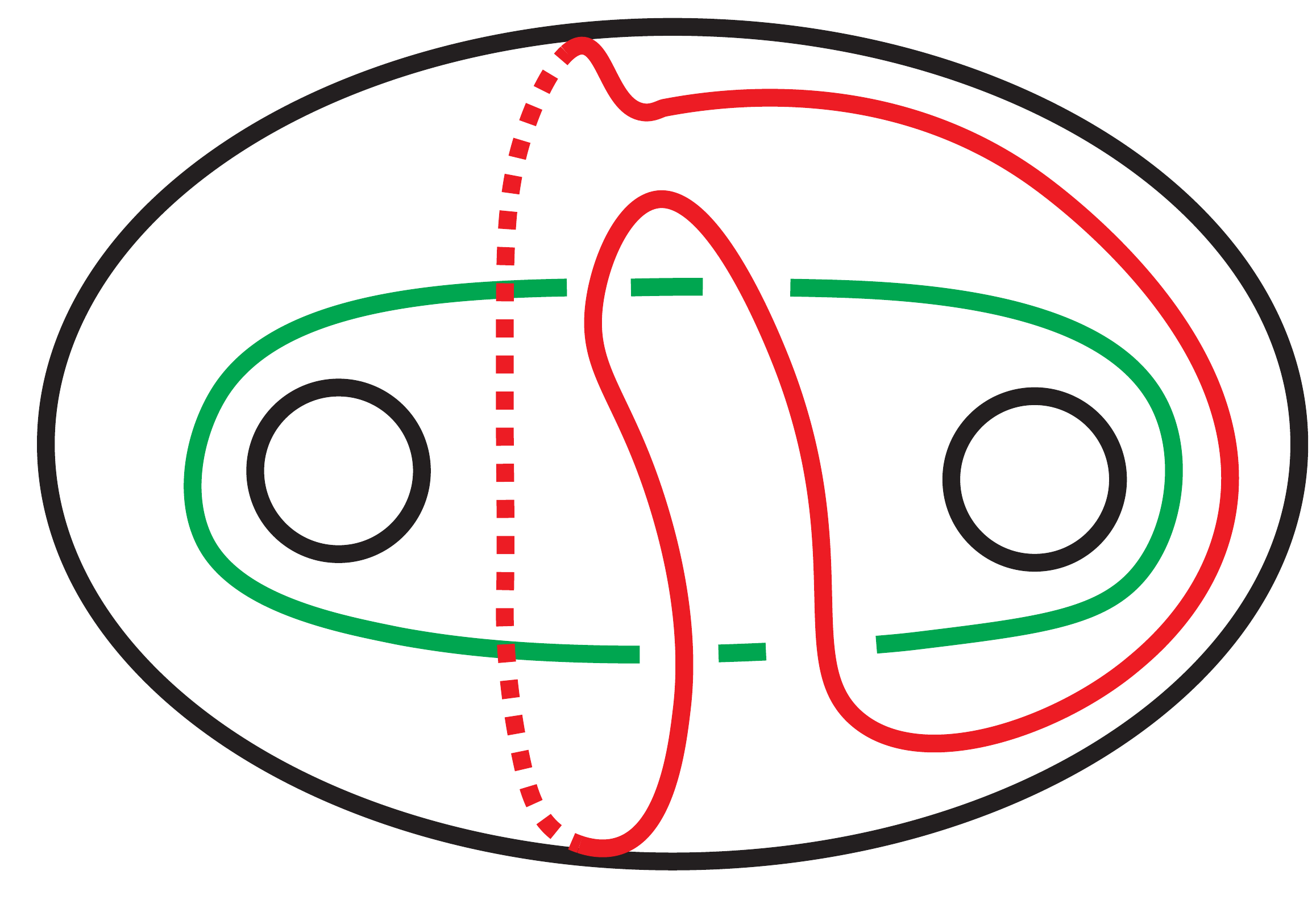}
\put(30, -9){$(a_{2}z)_{(2)}$}
\end{overpic}}}$
\vspace*{2mm}
    \caption{The curves $a_2z, (a_2z)_{(1)}$, and $(a_{2}z)_{(2)}$.}
    \label{fig: a2z}
\end{figure}

We see that  $(a_2 z)_{(1)}=x_2 z= (a_2 z)_{(2)}$. Hence, $w(a_2 z)=0 \in \cK_1$.
\epr
Before proceeding further, we discuss a consequence of Lemma \ref{r1}. Note that  $$\tau(G_k)=-G_k\in \fG.$$ 
As $\fG$ is generated by $G_k$ we conclude that $\tau(\fG)\subset \fG$.  From $\tau^2=\id$ we get $\tau(\fG)= \fG)$.
Since $\cK_1 = \fG$ by Lemma \ref{r1}, we have $\tau(\cK_1) = \cK_1$.
In particular, since $y^k - q^6 z_k \in \cK_1$, we have
\be y^k - q^{-6} \tau (z_k) \in \cK_1.
\label{eq.reflection}
\ee

\bpr[Proof that $\cK_4 \subset \cK_1$] The diagrams for the curves $a_4z, (a_4z)_{(1)}$, and $(a_{4}z)_{(2)}$ are illustrated in Figure \ref{fig: a2z}.

\begin{figure}[h]
    \centering
$\vcenter{\hbox{\begin{overpic}[scale=.075]{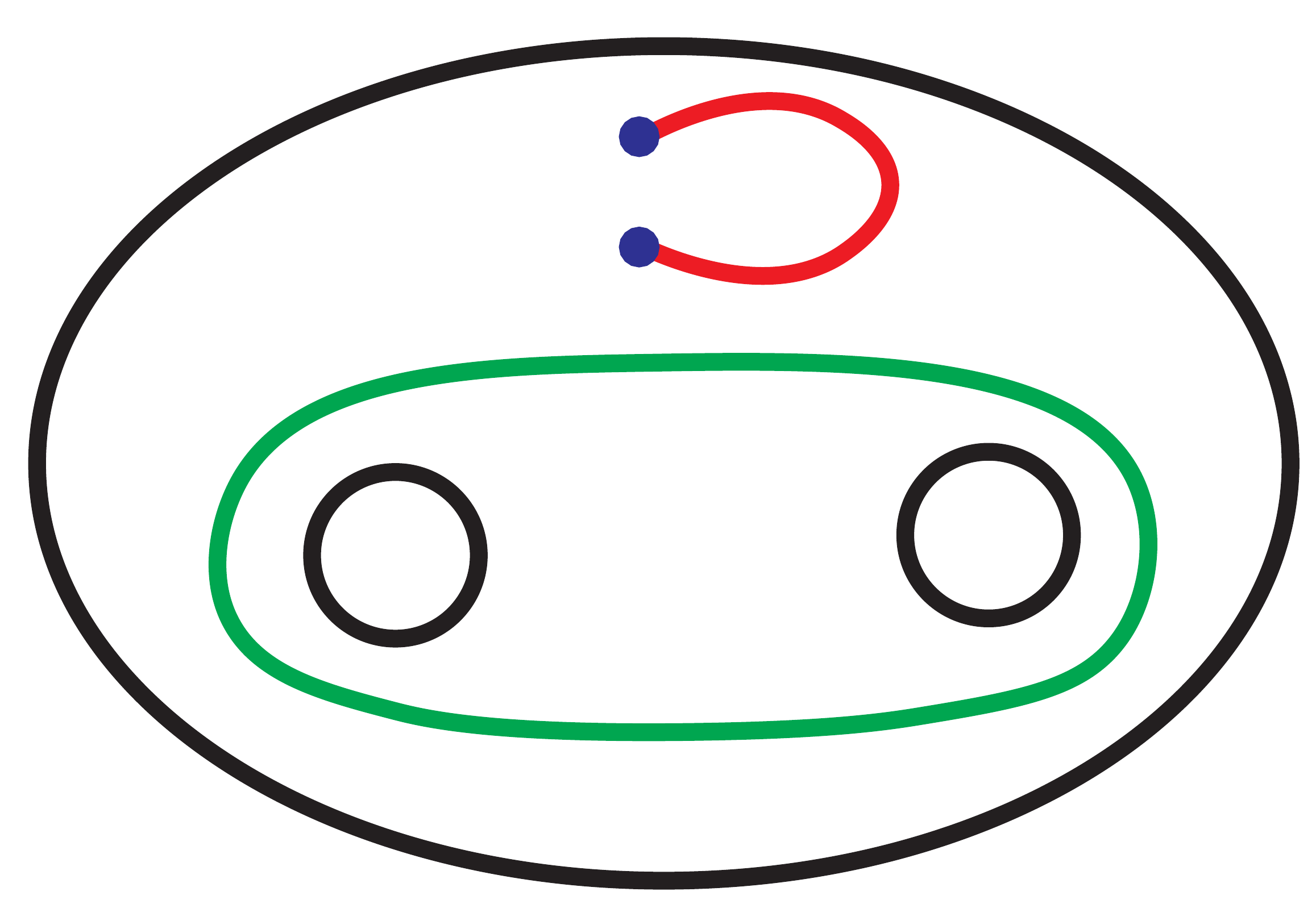}
\put(35, -8){$a_{4}z$}
\end{overpic}}} \hspace{4mm}
\vcenter{\hbox{\begin{overpic}[scale=.075]{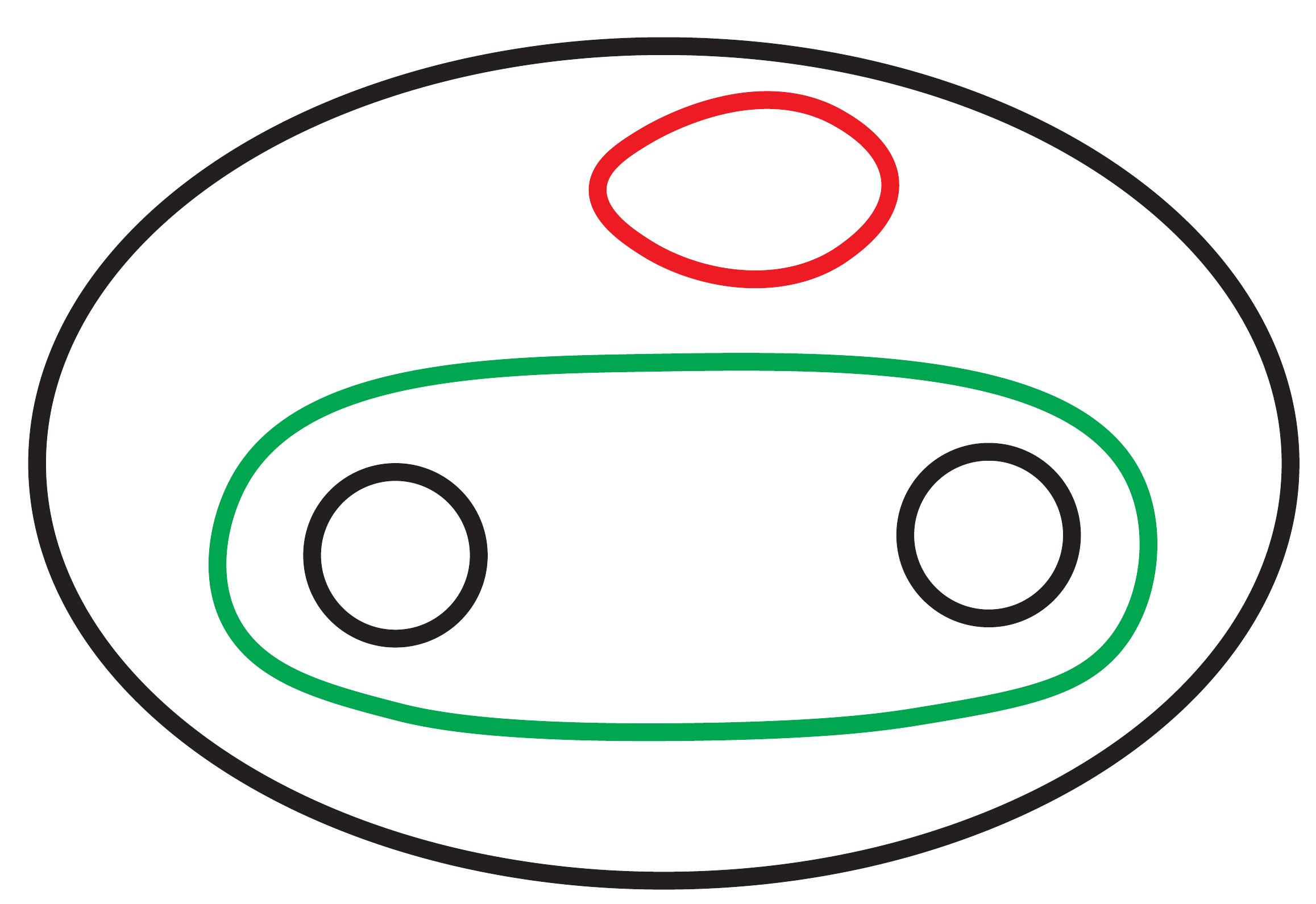}
\put(30, -9){$(a_{4}z)_{(1)}$}
\end{overpic}}} \hspace{4mm}
\vcenter{\hbox{\begin{overpic}[scale=.075]{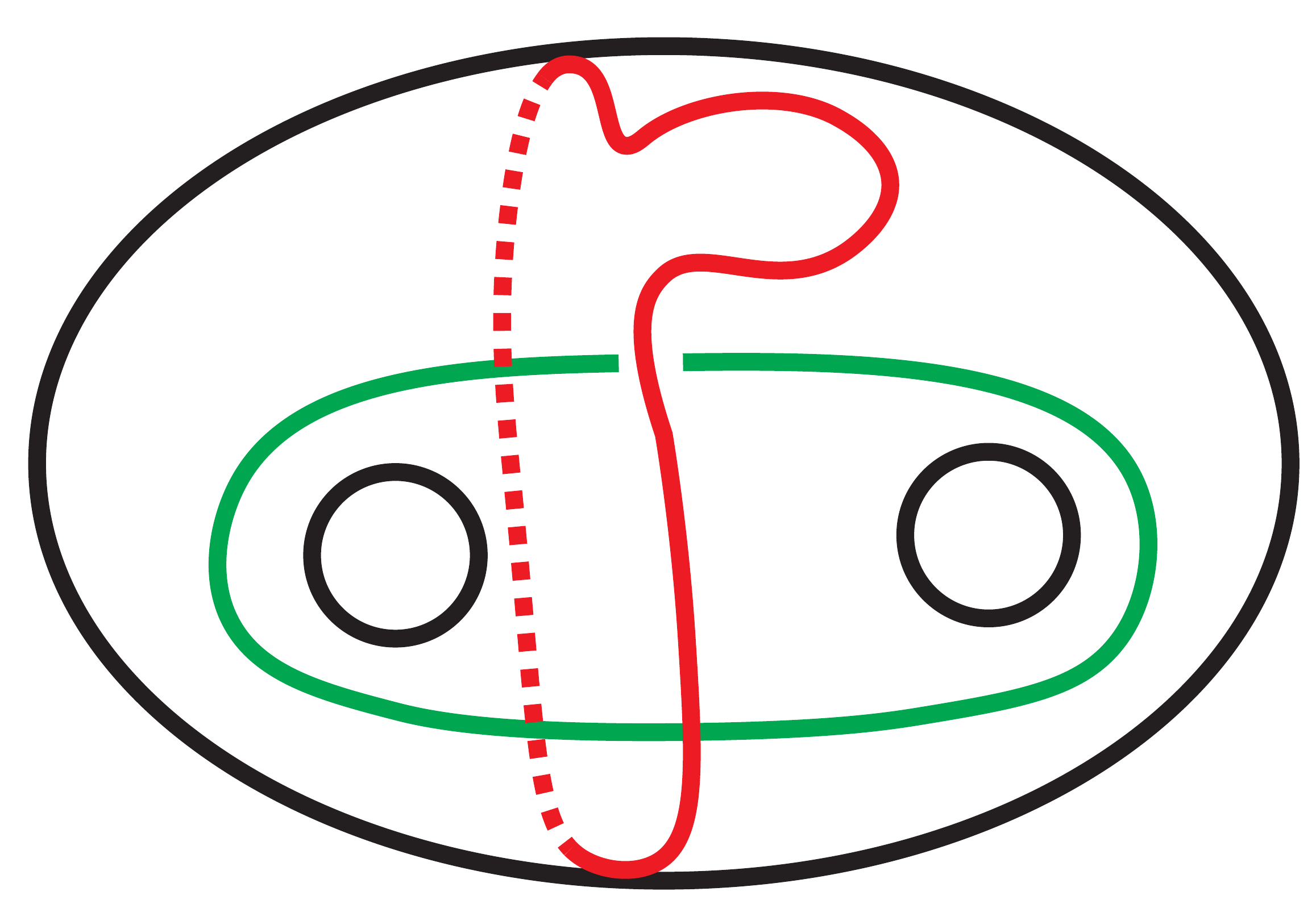}
\put(30, -9){$(a_{4}z)_{(2)}$}
\end{overpic}}}$
\vspace*{2mm}
    \caption{The curves $a_4z, (a_4z)_{(1)}$, and $(a_{4}z)_{(2)}$.}
    \label{fig: a4z}
\end{figure}

With $z=y^k$, we see
\begin{align}
(a_4 z)_{(1)} = (-q^2 -q^{-2}) y^k. 
\end{align}

  Using the skein relation at the top right  crossing, we get
 \be 
 (a_4 z)_{(2)} = q^{-1}\vcenter{\hbox{\begin{overpic}[scale=.075]{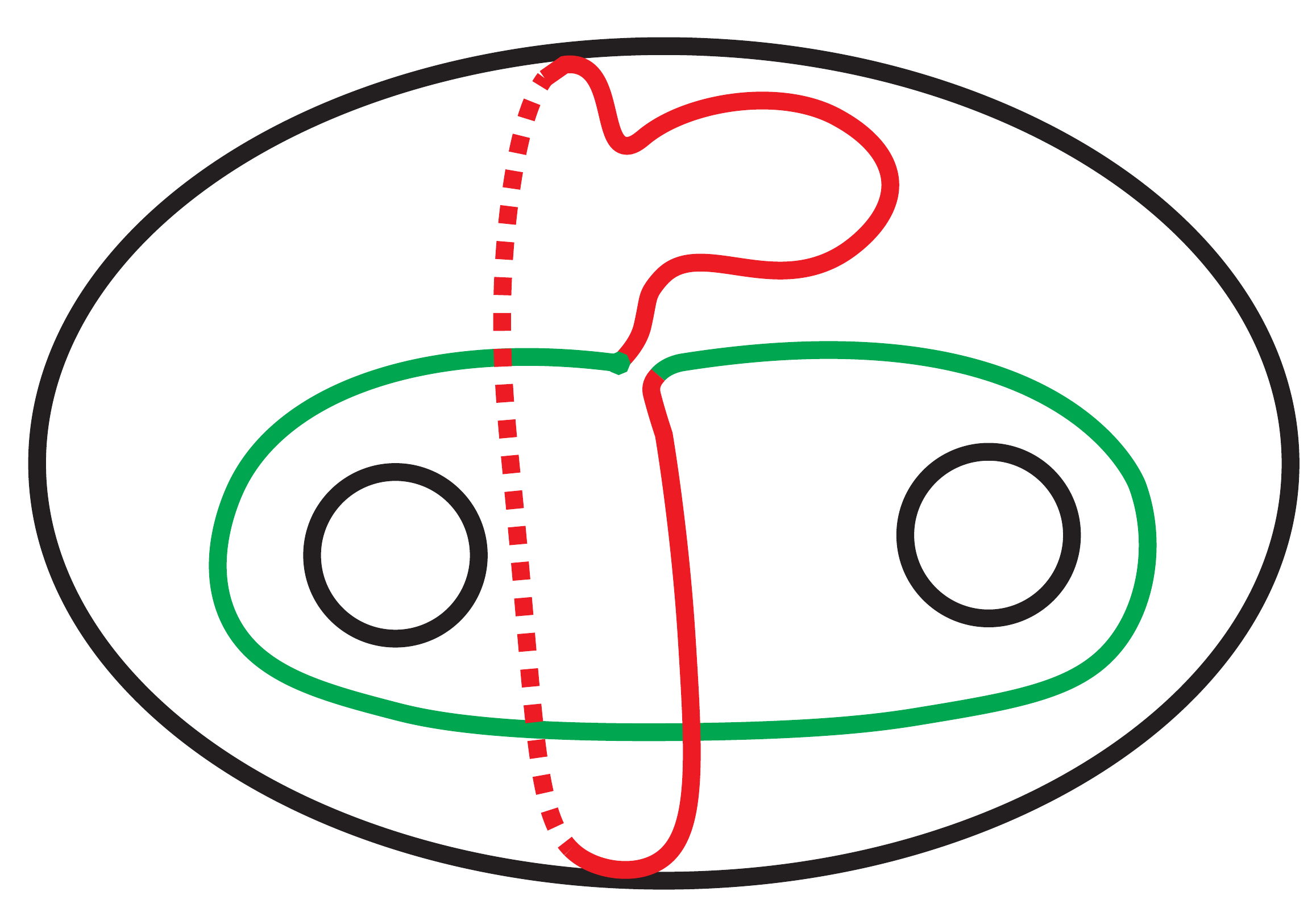}
\end{overpic}}} + q \vcenter{\hbox{\begin{overpic}[scale=.075]{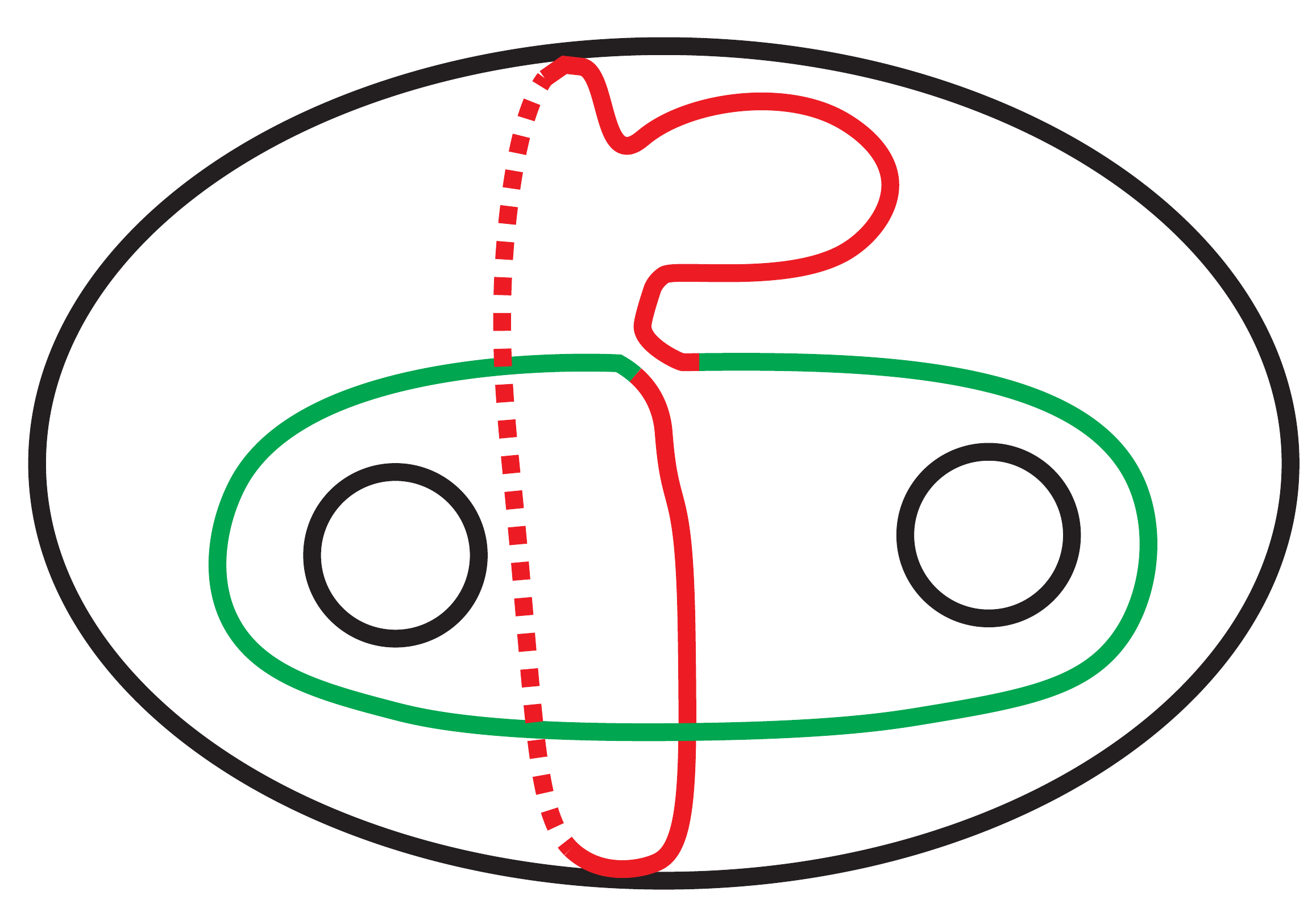}
\end{overpic}}}  = - q^{-4} \tau (z_k) -q^4 z_k .
 \ee 
From \eqref{eq.reflection} we get 
$$ (a_4 z)_{(2)}  = (-q^{-2}- q^2) y^k  \pmod {\cK_1} =   (a_4 z)_{(1)} \pmod {\cK_1} . \qedhere $$
\epr

\bpr[Proof that $\cK_3 \subset \cK_1$] See Figure \ref{fig: a3z} for illustrations of $a_3 z, (a_3 z)_{(1)}$ and $ (a_3 z)_{(2)}$.

\begin{figure}[h]
    \centering
$\vcenter{\hbox{\begin{overpic}[scale=.075]{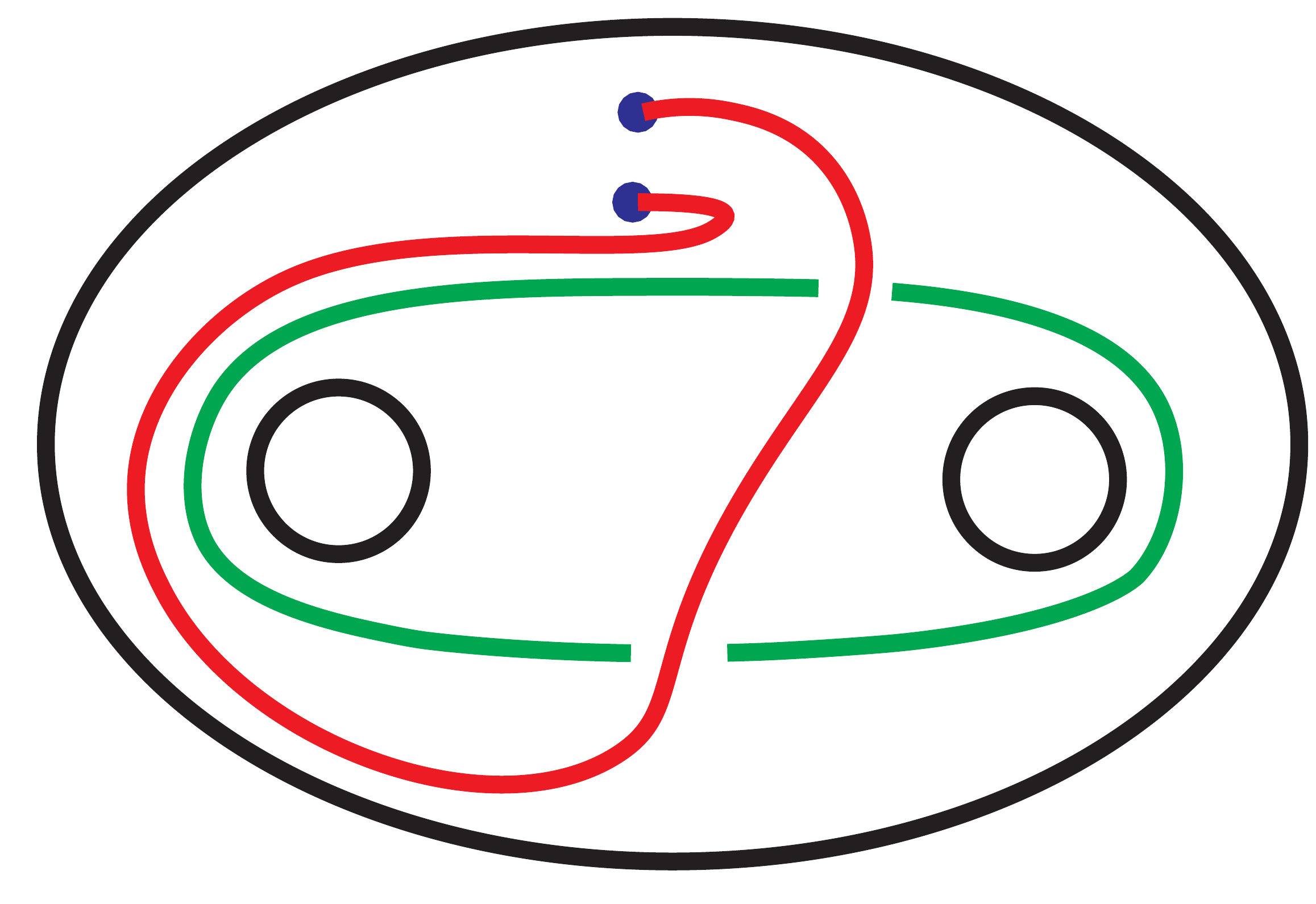}
\put(35, -8){$a_{3}z$}
\end{overpic}}} \hspace{4mm}
\vcenter{\hbox{\begin{overpic}[scale=.075]{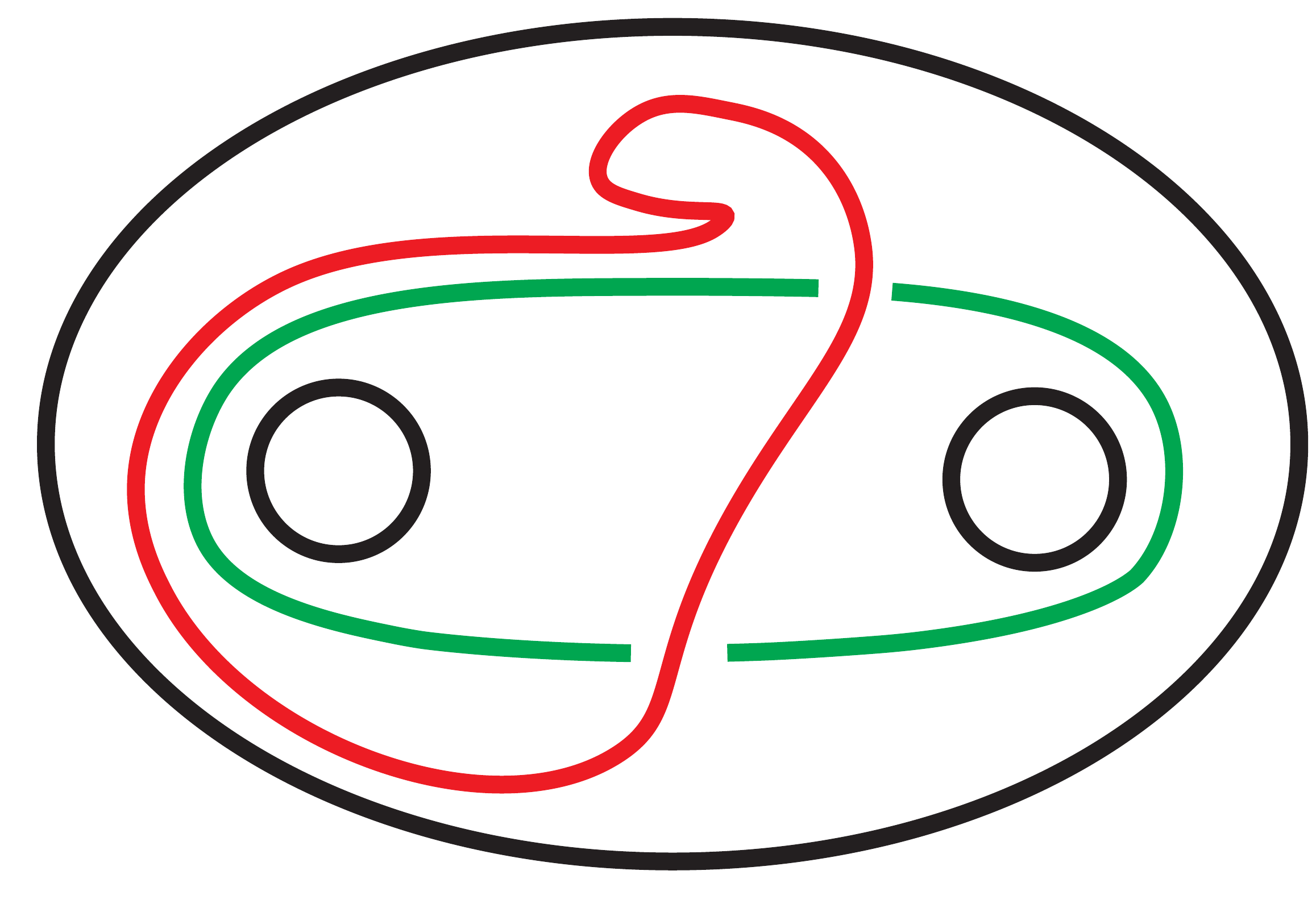}
\put(30, -9){$(a_{3}z)_{(1)}$}
\end{overpic}}} \hspace{4mm}
\vcenter{\hbox{\begin{overpic}[scale=.075]{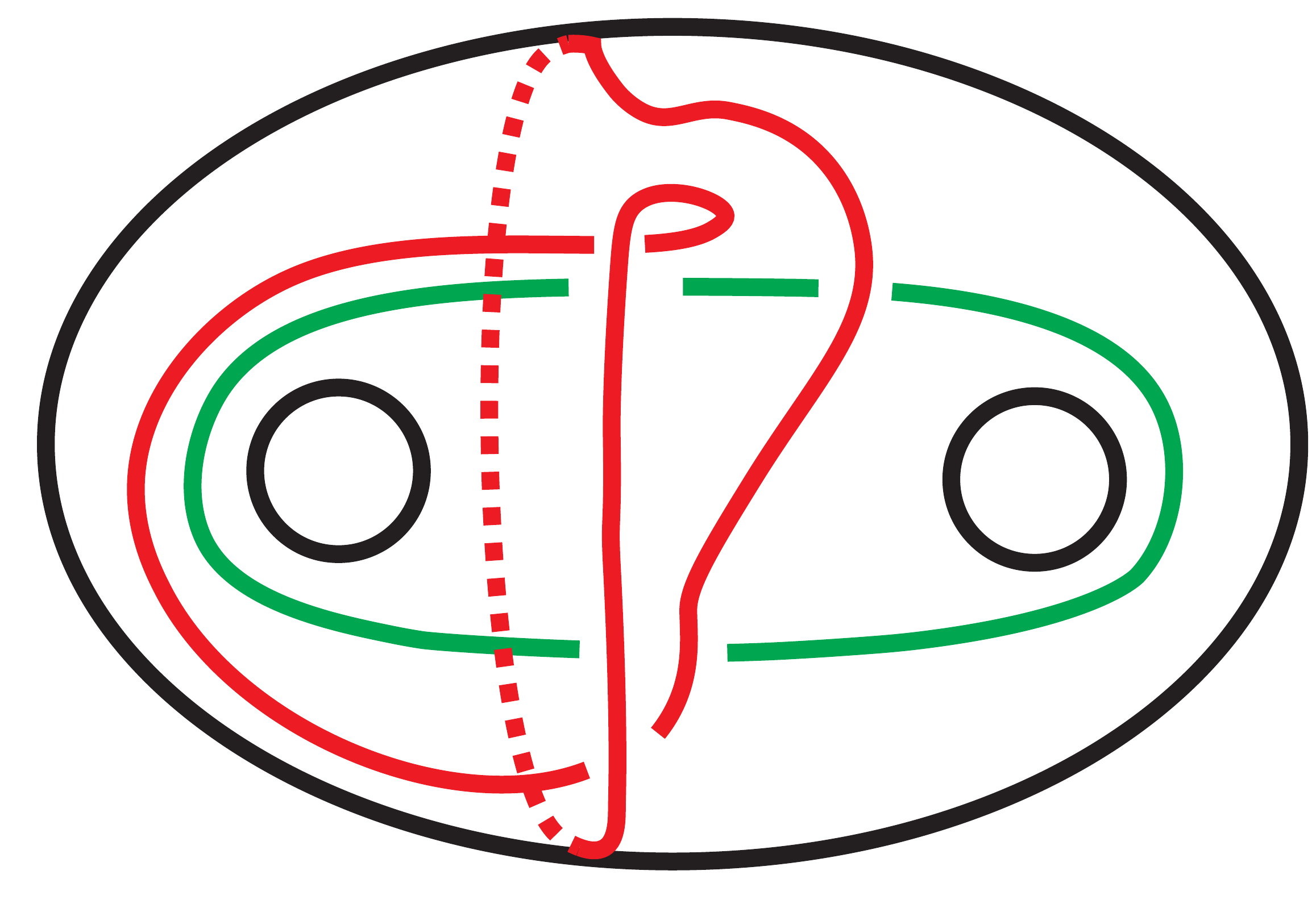}
\put(30, -9){$(a_{3}z)_{(2)}$}
\end{overpic}}}$
\vspace*{2mm}
    \caption{The curves $a_3z, (a_3z)_{(1)}$, and $(a_{3}z)_{(2)}$.}
    \label{fig: a3z}
\end{figure}

 We have $ (a_3 z)_{(1)} =  x_1 z$. Removing the kink in $(a_{3}z)_{(2)}$ and then resolving the red crossing, we get

\begin{align}
 (a_3 z)_{(2)} & = -q^{3}\vcenter{\hbox{\begin{overpic}[scale=.075]{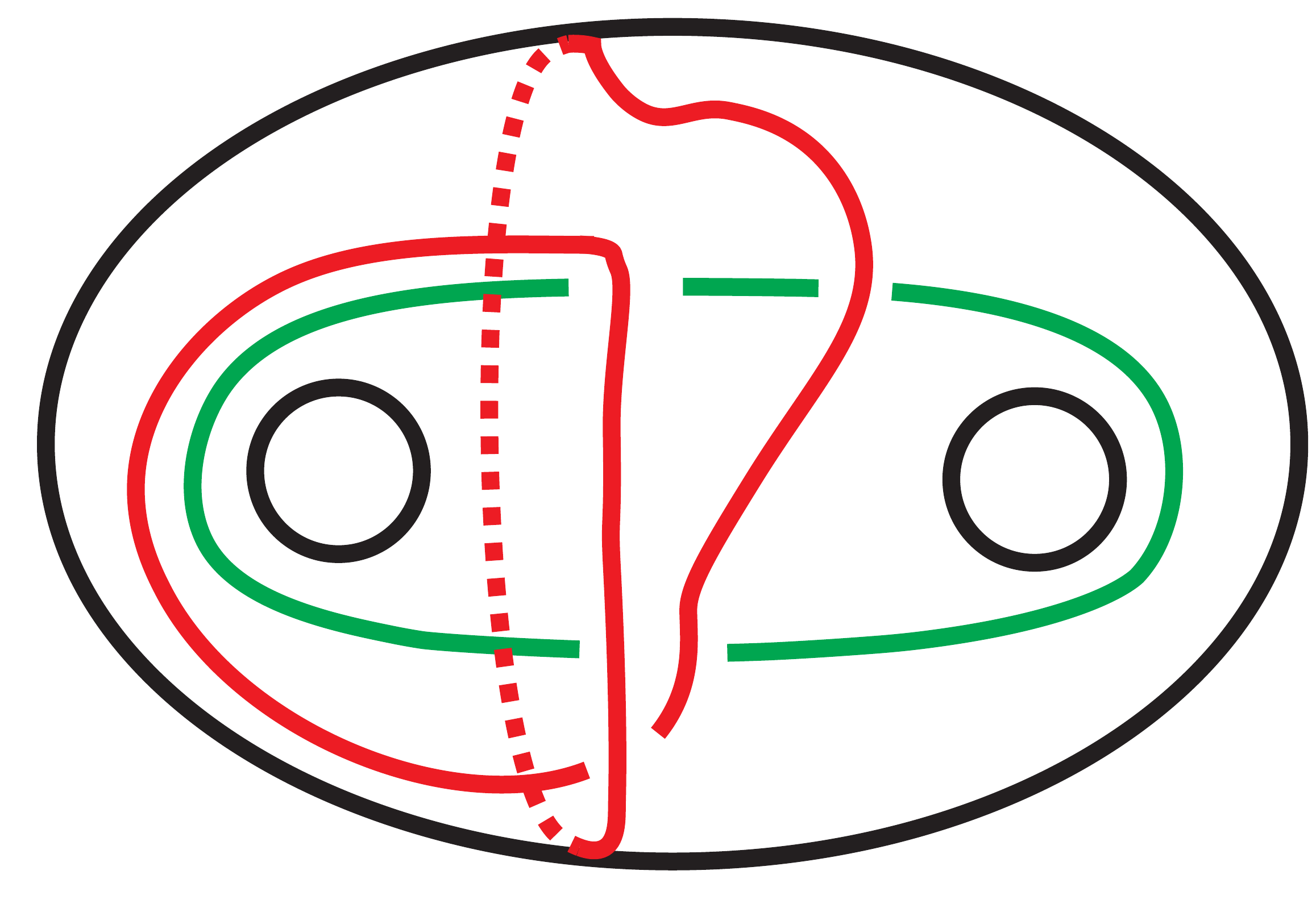}
\end{overpic}}} = -q^4 \vcenter{\hbox{\begin{overpic}[scale=.075]{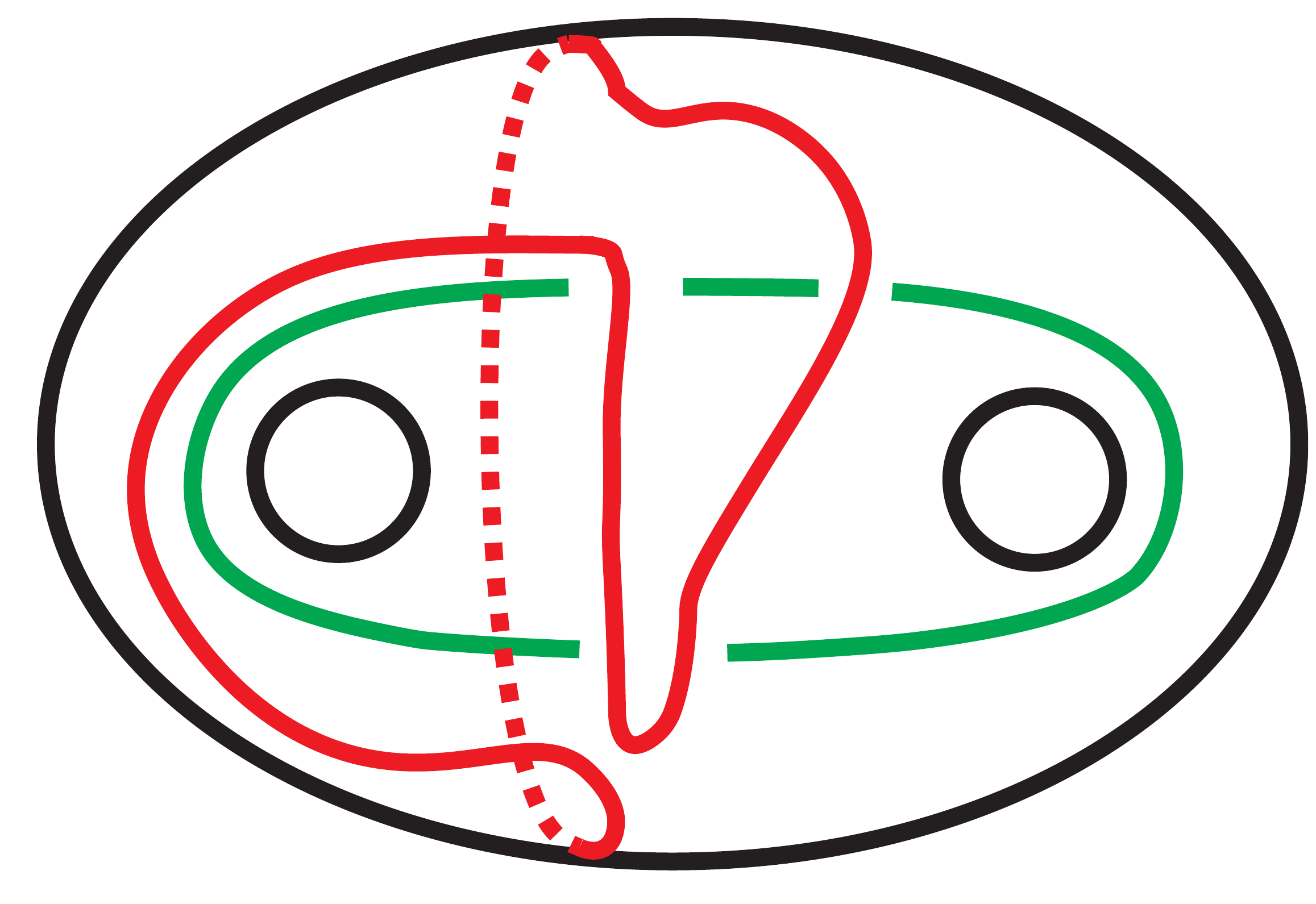}
\end{overpic}}}  -q^2 \vcenter{\hbox{\begin{overpic}[scale=.075]{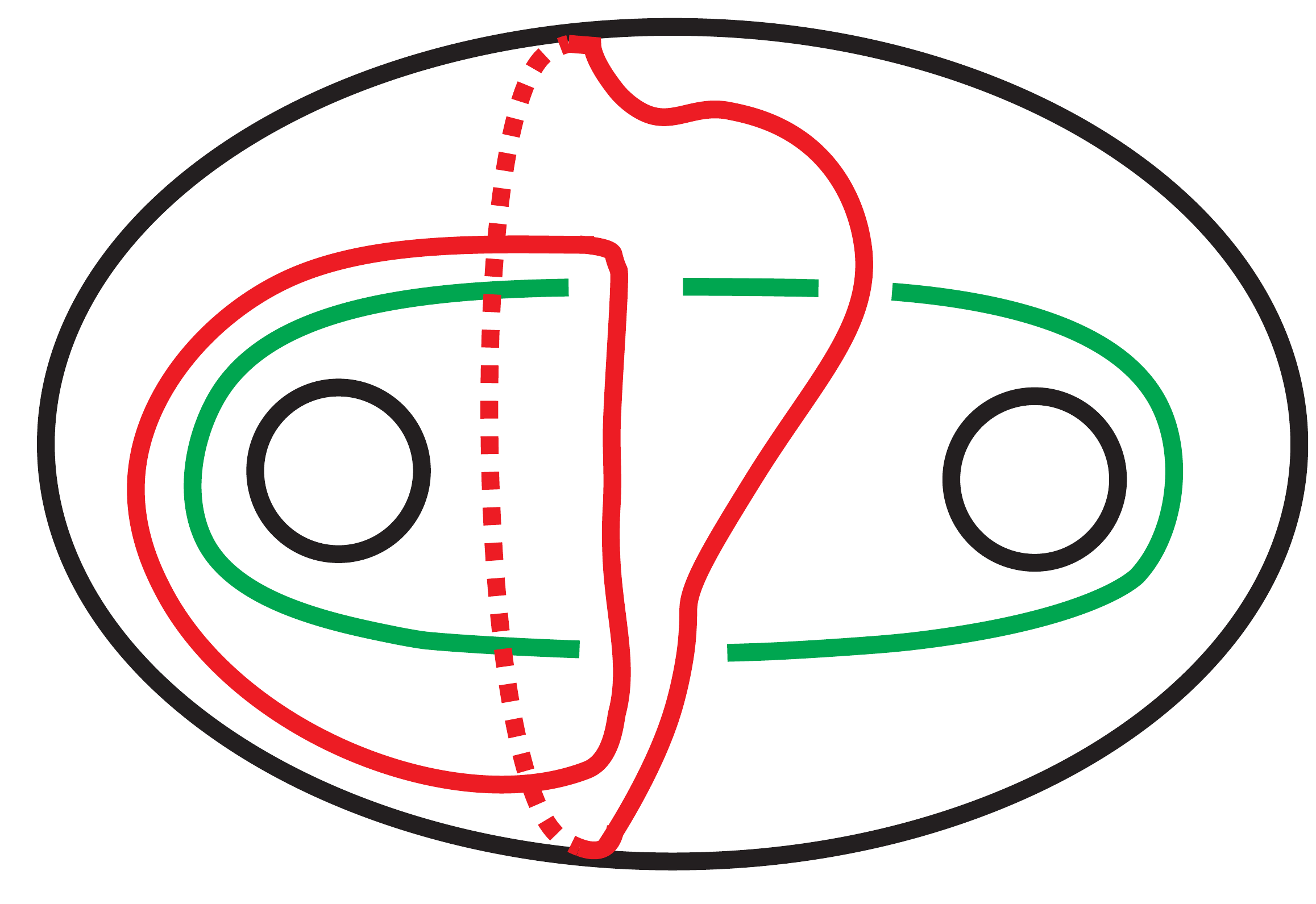}
\end{overpic}}}\\ & = - q^4 x_1 z  - q^2 (-q^2 - q^{-2} x_1 z \pmod {\cK_3}\\
& = x _1 z \pmod {\cK_3} =  x _1 z \pmod {\cK_1}= (a_3 z)_{(1)} \pmod {\cK_1}. \qedhere
\end{align}
\epr

\subsection{Proof of Lemma \ref{r1}}

\def\hR{{\hat R}}
\def\cI{\mathcal I}
\def\hT{{\hat T}}

\begin{proof} The first and second type Chebyshev polynomials, $T_n(x)$ and $S_n(x)$, respectively, can be defined by requiring that if $x= u+ u^{-1}$  then
$$ T_n(x) = u^n + u^{-n} \ \text{and} \ S_n(x) = \frac{u^{n+1}-u^{-n-1}}{u-u^{-1}}.$$
Note that the definitions are valid for all integers $n\in \BZ$. We have 
\begin{align}
T_{-n} &= T_n, \ S_{-n-1} = - S_{n-1} \\
T_1 T_n & =  S_{n+1} - S_{n-3}. 
\end{align}

Note that $T_0(x)=2$, while $T_n(x)$ is a monic polynomial of degree $n$ for $n \ge 1$. Let $\hT_n:= T_n$ for $n \neq 0$ and $\hT_0=1$. 
 This modification guarantees that $\{ \hT_n(y), n \ge 0\}$ is a basis of $\hR[y]$ over $\hR$. 
Hence, $\{w(a_1* \hT_n(y)), n \ge 0\}$ spans $K_1$ over $\hR$. See Figure \ref{fig: a1*Tn} for an illustration. 

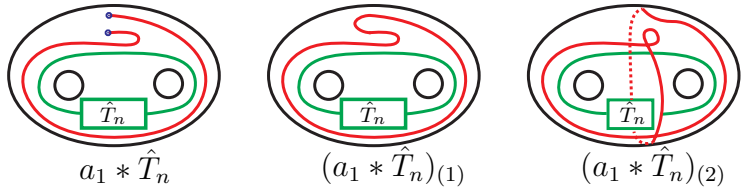
\begin{figure}[h]
    \centering
$\vcenter{\hbox{\begin{overpic}[scale=.075]{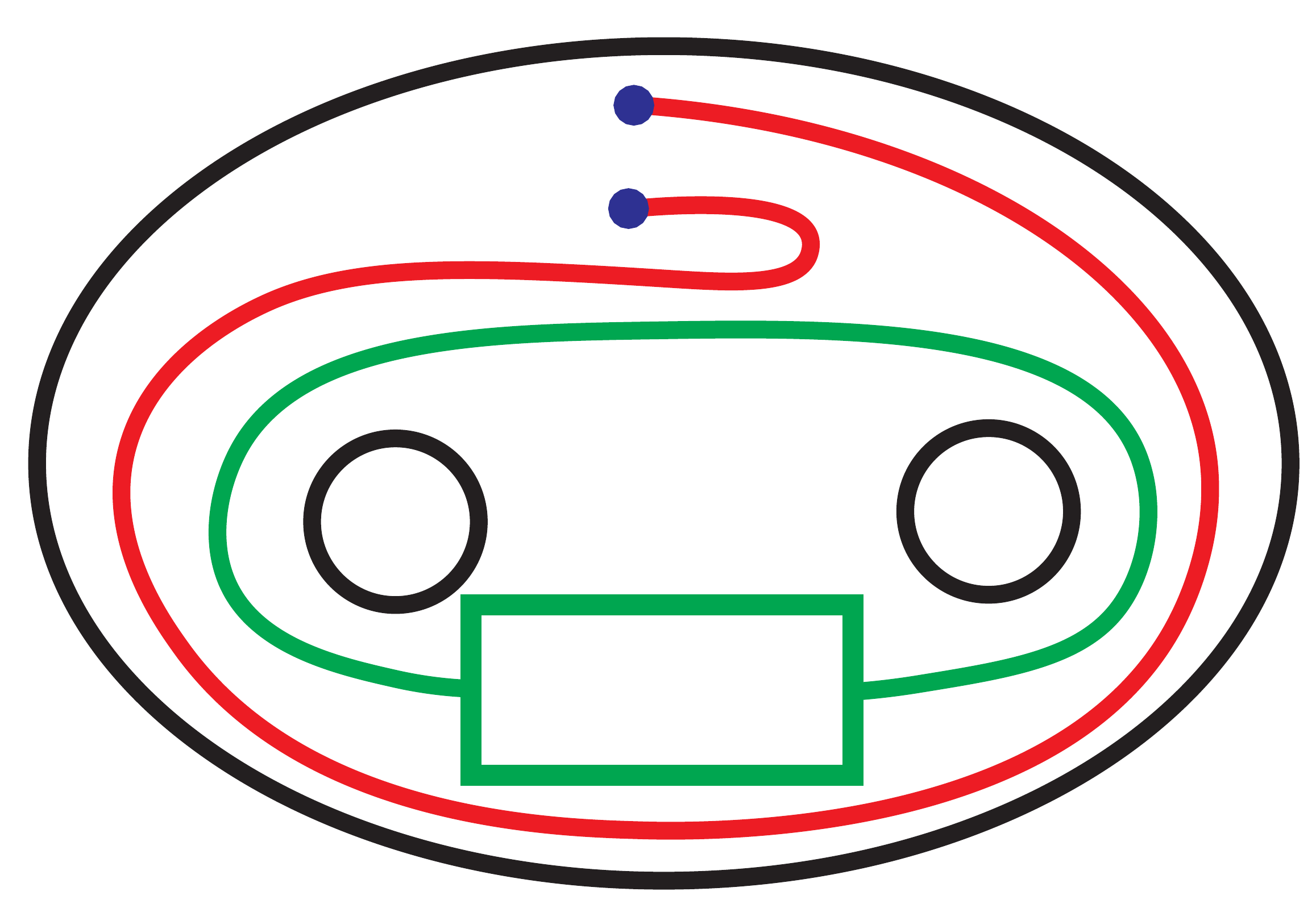}
\put(29, -11){$a_{1}* \hat T_n$}
\put(37, 10.5){\tiny{$\hat T_n$}}
\end{overpic}}} \hspace{4mm}
\vcenter{\hbox{\begin{overpic}[scale=.075]{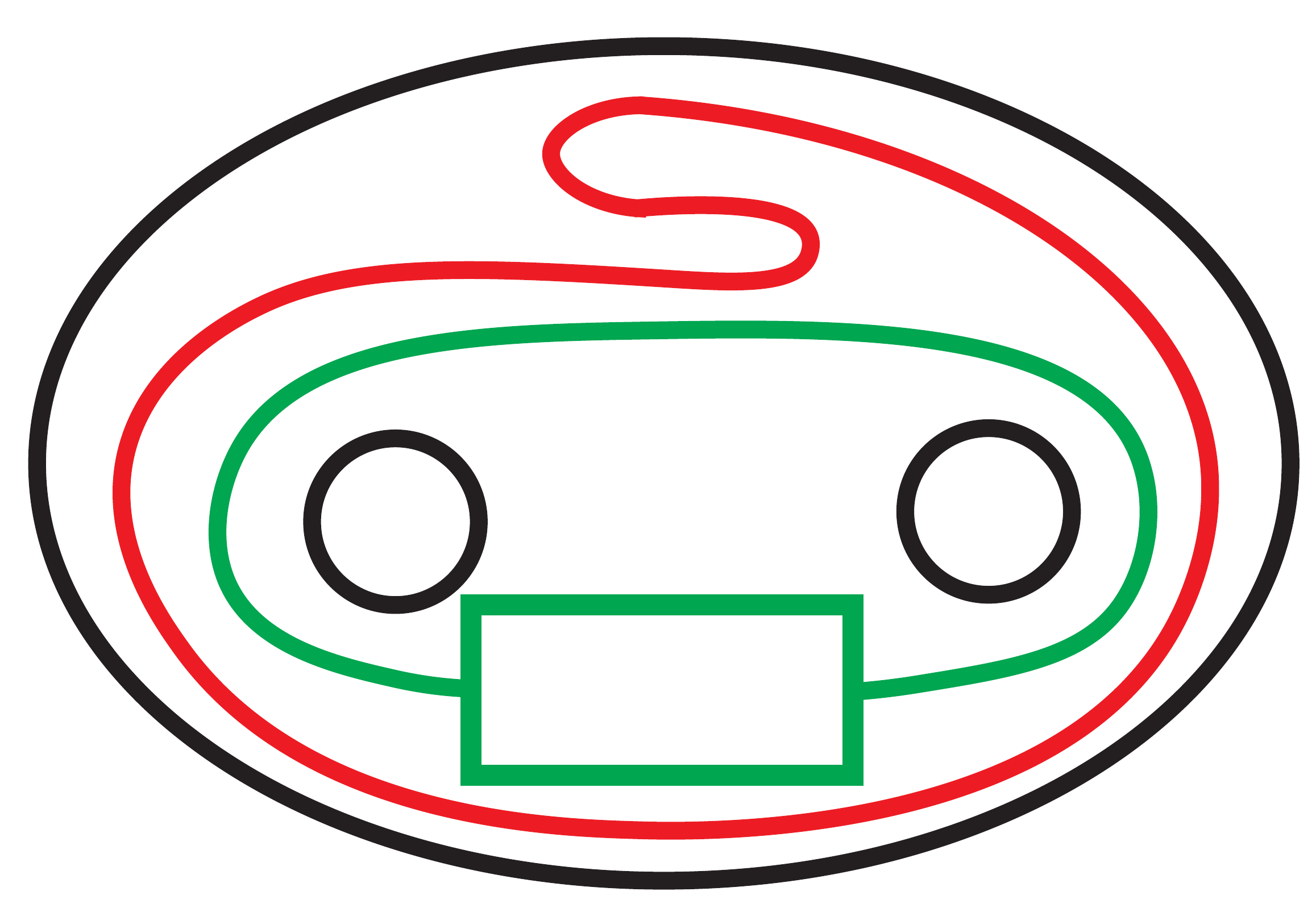}
\put(22, -10){$(a_{1}* \hat T_n)_{(1)}$}
\put(37, 10.5){\tiny{$\hat T_n$}}
\end{overpic}}} \hspace{4mm}
\vcenter{\hbox{\begin{overpic}[scale=.075]{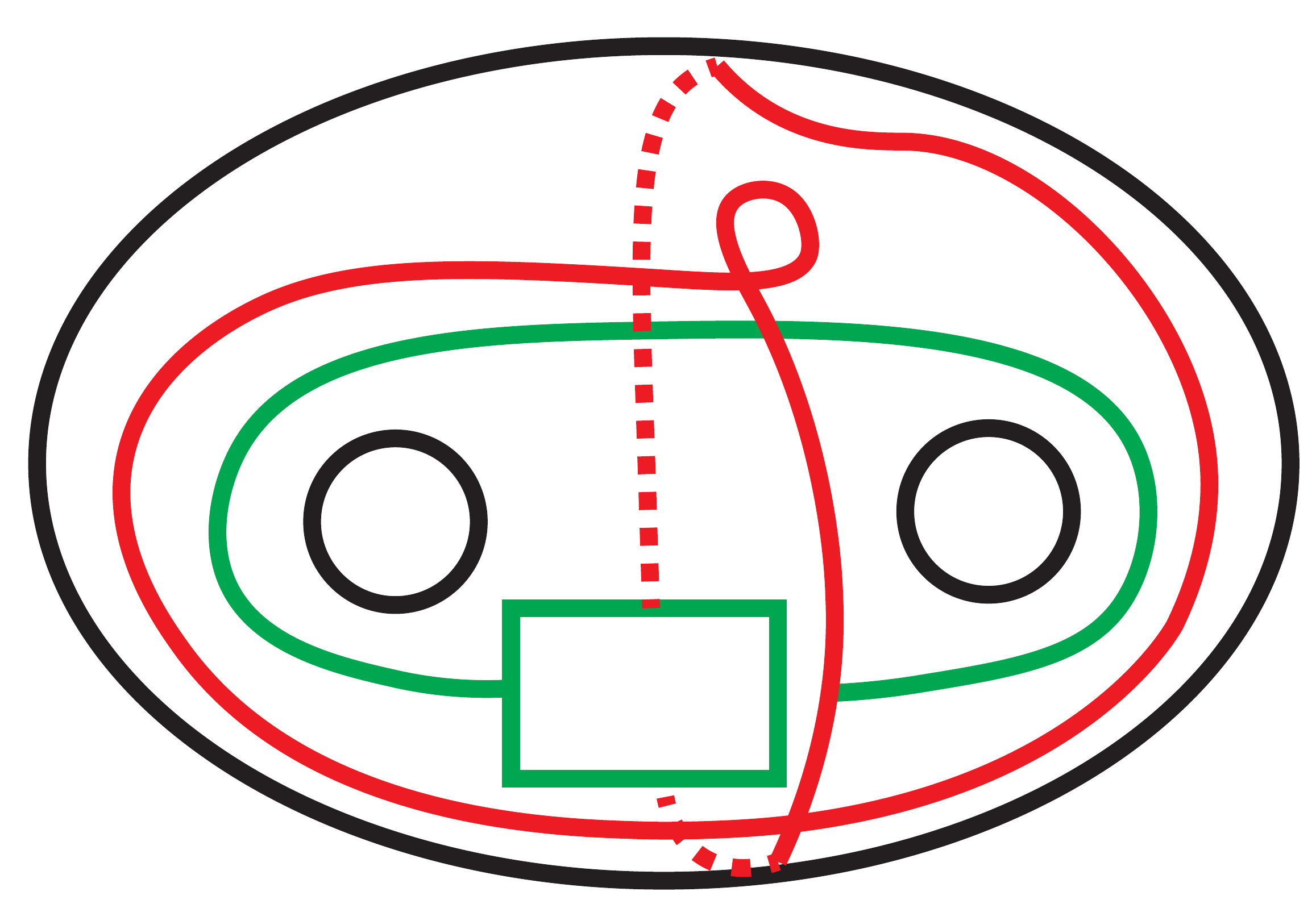}
\put(22, -10){$(a_{1}* \hat T_n)_{(2)}$}
\put(37, 10.5){\tiny{$\hat T_n$}}
\end{overpic}}}$
\vspace*{2mm}
    \caption{The curves $a_{1}* \hat T_n, (a_{1}* \hat T_n)_{(1)}$ and $(a_{1}* \hat T_n)_{(2)}$.}
    \label{fig: a1*Tn}
\end{figure}

We have
\begin{align}
(a_1* \hT_n(y))_{(1)}  &= y T_n(y) \\
(c_2* \hT_n(y))_{(1)}  &=  q^6 \sigma_n, \text{where} \ \sigma_n = 
    \centering
\vcenter{\hbox{\begin{overpic}[scale=.075]{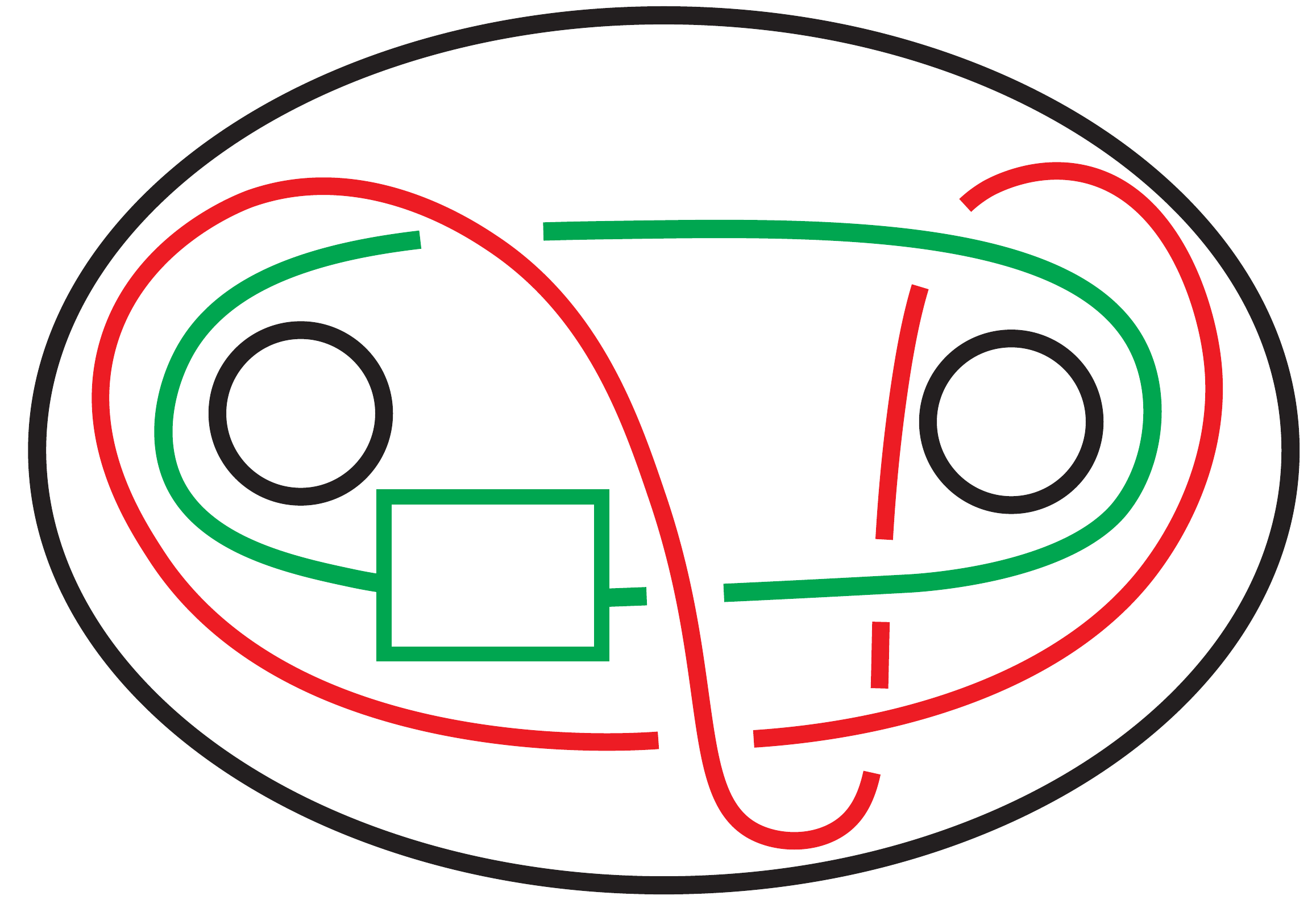}
\put(27, 18){\tiny{$\hat T_n$}}
\end{overpic}}}
\end{align}

For $n \ge 1$ let  
\begin{equation}\label{eqn:fn}
 F_n := -q^{-2n -4} \, w( a_1 * \hT_{n-1}(y)= q^{-2n +4}[\sigma_{n-1} - q^{-6}y\hT_{n-1}(y)].  
\end{equation} 
Then $\{ F_n, n \ge 1 \}$ spans $K_1$ over $\hR$.   We will show  $\cK_1=\fG$  by showing  that
$\{F_n , n \geq 1\}$ and $\{ G_n, n \geq 1 \}$ are related by a lower unitriangular matrix.

Extend the definition of $\{k\}$ and $G_k$ to all integer values of $k$ by the same formulas:
\begin{equation}\label{eqn:jn}
 \{k\} = q^{2k} - q^{-2k}, \quad    G_k :=  \{k+1\}S_k(y) + (-1)^{(k+1)}\{1\}S_k(x_1)S_k(x_2) \in \hat R[y]. 
\end{equation}
Note that $ G_{k-1} =  G_{-k-1}$ and $ G_0 =  G_{-1} =  G_{-2} = 0$. Also, $deg_y( G_k)=k$, for $k \geq 1$.

Define the following elements of $\fG$:
\begin{align}
Q_n &:= \frac{1}{2}\sum \limits_{k=0}^n   G_{n-2k-1} = \frac{1}{2}( G_{n-1} +  G _{n-3} + \hdots +  G _{-n-1}), \label{eq.Qn} \\
U_n  & :=  G_n  + x_1 x_2  G_{n-1}  -  G_{n-2}  + (x_1^2 + x_2 ^2)Q_{n-1}  + 2x_1x_2Q_{n-2}. \label{eq.Un}
\end{align}
Then both $U_n$ and $ G_n$ are in $\fG$, and both have the same degree $n$ and the same leading term,  i.e. $ \deg_y(U_n -G_n ) \le n-1$. Hence $\{U_n , n \geq 1\}$ and $\{ G_n, n \geq 1 \}$ are related by a lower unitriangular matrix. It follows that $\{U_n , n \geq 1\}$ spans $\fG$  over $\hR$.
\begin{lemma}\label{r.maintech}

We have $F_1 =G_1$ and for $n\geq 2$,

\begin{equation}\label{eqn:lemma}
F_n = U_n - q^{4-4n}U_{n-2}.    
\end{equation}
\end{lemma}

The lemma shows that $\{F_n , n \geq 1\}$ and $\{ U_n, n \geq 1 \}$ are related by a lower unitriangular matrix. Hence, the $\hR$-span of $\{ U_n, n \geq 1 \}$ is equal to the $\hR$-span of $\{ F_n, n \geq 1 \}$, or $\fG= \cK_1$.

We now prove Lemma \ref{r.maintech}.

Let us write $ G_k = G'_k +  G''_k$, where
$$ G'_k:= \{k+1\}S_k(y) \ \text{and} \  G''_k := (-1)^{k+1}\{1\}S_k(x_1)S_k(x_2).$$

For  $\varepsilon \in \{ ', ''\}$, define  $Q^\varepsilon_n $ and $U^\varepsilon_n$ using the same formulas \eqref{eq.Qn} and \eqref{eq.Un}, with every $G,Q,U$ replaced by $G^\varepsilon, Q^\varepsilon, U^\varepsilon$, respectively. Then
\be  U_n = U'_n + U''_n.
\ee

\blem \label{r.222}
For $n \ge 1$ we have
\be F_{n+1} = U'_{n+1} - q^{-4n} U'_{n-1}. 
\ee
\elem
\bpr

In \cite{EG},
 Esebre and Gelca  showed that in $\cS(D_2)$,   
 \no{   \begin{equation}\label{gelcaesebreformula}
    \begin{split}
\sigma_n & = q^{4n+2}S_{n+1}(y)+q^{-4n-2}S_{n-1}(y) - q^{4n-2}S_{n-1}(y)  \\
 & + q^{4n} x_1 x_2 S_{n}(y) - q^{-4n} x_1 x_2 S_{n-2}(y)   - q^{-4n+2}S_{n-3}(y)
- q^{-4} x_1 x_2 S_{n}(y)  + q^{-4} x_1 x_2 S_{n-2}(y) \\ 
& + q^{-2}(q^{4n} - 1) (x_1^2 + x_2^2)\sum \limits_{k = 0}^{n-1}q^{\red{-4k}}S_{n-2k-1}(y)
+ 2q^{-4}(q^{4n} - 1) x_1 x_2 \sum \limits_{k = 0}^{n-1}q^{\red{-4k}}S_{n-2k-2}(y).
\end{split}
\end{equation}
}
\begin{equation}
    \begin{split}
\sigma_n & = q^{4n+2}S_{n+1}(y)  +( q^{4n}- q^{-4}) x_1 x_2 S_{n}(y) + (q^{-4n-2}  - q^{4n-2} )S_{n-1}(y)    \\
 & +( q^{-4}- q^{-4n}) x_1 x_2 S_{n-2}(y) - q^{-4n+2}S_{n-3}(y)
 \\ 
& + q^{-2}(q^{4n} - 1) (x_1^2 + x_2^2)\sum \limits_{k = 0}^{n-1}q^{-4k}S_{n-2k-1}(y)
+ 2q^{-4}(q^{4n} - 1) x_1 x_2 \sum \limits_{k = 0}^{n-1}q^{-4k}S_{n-2k-2}(y).
\end{split}
\end{equation}
With this result, the proof of the lemma is a direct calculation, the details of which are provided below.

Using $F_{n+1} = q^{-2n+2}(\sigma_n - q^{-6}y T_n)$ with $yT_n = S_{n+1} - S_{n-3}$, we get
 \begin{align*}
        F_{n+1} & = G'_{n+1}(y)  + x_1 x_2 G'_{n}(y) -(1+ q^{-4n})G'_{n-1}(y) +  q^{-4n}G'_{n-3}(y)   +q^{-4n} x_1 x_2 G'_{n-2}(y) 
 \\ 
& + \{n\} (x_1^2 + x_2^2) ( q^{-4n} S_{n-1} + q^{-2n} Q'_n  )
+ 2q^{-2n}\{n\} x_1 x_2 Q'_{n-1} \\
& = G'_{n+1}(y)  + x_1 x_2 G'_{n}(y) -G'_{n-1}(y) + (x_1^2 + x_2^2)  Q'_n  
+2 x_1 x_2 Q'_{n-1}
 \\ 
& - q^{-4n} \left[ G'_{n-1}(y) +  (x_1^2 + x_2^2)  G'_{n-2} 
+ 2  x_1 x_2 Q'_{n-1} - G'_{n-3}(y)   - x_1 x_2 G'_{n-2}(y) 
\right] \\
& =  G'_{n+1}(y)  + x_1 x_2 G'_{n}(y) -G'_{n-1}(y) + (x_1^2 + x_2^2)  Q'_n  
+2 x_1 x_2 Q'_{n-1}
 \\ 
& - q^{-4n} \left[ G'_{n-1}(y) +  x_1 x_2 G'_{n-2}(y)- G'_{n-3}(y) + (x_1^2 + x_2^2)  G'_{n-2} 
+ 2  x_1 x_2 Q'_{n-1}\right] \\ 
&  +q^{-4n} 2x_1 x_2 G'_{n-2}(y) \\
& = G'_{n+1}(y)  + x_1 x_2 G'_{n}(y) -G'_{n-1}(y) + (x_1^2 + x_2^2)  Q'_n  
+2 x_1 x_2 Q'_{n-1}
 \\ 
& - q^{-4n} \left[ G'_{n-1}(y) +  x_1 x_2 G'_{n-2}(y)- G'_{n-3}(y) + (x_1^2 + x_2^2)  G'_{n-2} 
+ 2  x_1 x_2 Q'_{n-3}\right] \\
& = U'_{n+1} - q^{-4n} U'_{n-1}.
    \end{align*}
\epr

\no{
 \begin{equation}\label{gelca2}
    \begin{split}
\sigma_{n-1} & = q^{4n-2}S_{n}(y)+q^{-4n+2}S_{n-2}(y) - q^{4n-6}S_{n-2}(y) - q^{-4n+6}S_{n-4}(y)  \\
 & + q^{4n-4} x_1 x_2 S_{n-1}(y) - q^{-4n+4} x_1 x_2 S_{n-3}(y) 
- q^{-4} x_1 x_2 S_{n-1}(y)  + q^{-4} x_1 x_2 S_{n-3}(y) \\ 
& + q^{-2}(q^{4n-4} - 1) (x_1^2 + x_2^2)\sum \limits_{k = 0}^{n-2}q^{\red{-4k}}S_{n-2k-2}(y)
+ 2q^{-4}(q^{4n-4} - 1) x_1 x_2 \sum \limits_{k = 0}^{n-2}q^{\red{-4k}}S_{n-2k-3}(y).
\end{split}
\end{equation}
}

\no{

$$U_n' := \mathcal G_n' + x_1 x_2 \mathcal G_{n-1}' - \mathcal G_{n-2}' + (x_1^2 + x_2 ^2)Q_{n-1}' + 2x_1x_2Q_{n-2}'.$$

$$U_{n+1}' := \mathcal G_{n+1}' + x_1 x_2 \mathcal G_{n}' - \mathcal G_{n-1}' + (x_1^2 + x_2 ^2)Q_{n}' + 2x_1x_2Q_{n-1}'.$$
$$U_{n-1}' := \mathcal G_{n-1}' + x_1 x_2 \mathcal G_{n-2}' - \mathcal G_{n-3}' + (x_1^2 + x_2 ^2)Q_{n-2}' + 2x_1x_2Q_{n-3}'.$$

$$ U_{n+1}' - q^{-4n} U_{n-1}' = \mathcal G_{n+1}' + x_1 x_2 \mathcal G_{n}' - \mathcal G_{n-1}' + (x_1^2 + x_2 ^2)Q_{n}' + 2x_1x_2Q_{n-1}'$$
 $$ -q^{-4n} (  \mathcal G_{n-1}' + x_1 x_2 \mathcal G_{n-2}' - \mathcal G_{n-3}' + (x_1^2 + x_2 ^2)Q_{n-2}' + 2x_1x_2Q_{n-3}'  )$$
 $$= \mathcal G_{n+1}' + x_1 x_2 \mathcal G_{n}' -( 1 + q^{-4n})\mathcal G_{n-1}' - q^{-4n} x_1 x_2 G'_{n-2} + q^{-4n} G'_{n-3}  $$
 }

\def\tR{{\tilde R}}

\begin{lemma}\label{r.333}
For all $n \geq 0$, we have $U''_n = 0$.
    \end{lemma}

\proof Note that $G''_n/\{1\}$ is in the ring $\BZ[x_1, x_2]$, which embeds into $\tR= \BZ[ u_1^{\pm 1}, u_2^{\pm 1}]$ via $x_1= u_1 +u_1^{-1}$ and $x_2=  u_2 + u_2^{-1}$.
The sequence $( G''_n/\{1\})_{n=0}^\infty$ satisfies a homogeneous linear recurrence relation with constant $\tR$-coefficients of degree $4$, with four eigenvalues $u_1^{\pm 1}$ and $u_2^{\pm 1}$. Hence $(U''_n/\{1\})$ also satisfies a degree $4$ recurrence relation. A quick calculation shows that $U''_n/\{1\} = 0$ for $n = 0,1,2$, and $3$.
This shows that $U''_n =0$ for all $n\geq 0$. 
\end{proof}
As $U_n = U'_n + U''_n$, from Lemmas \ref{r.333} and \ref{r.222} we get Lemma \ref{r.maintech}. This completes the proof of Theorem \ref{thm1a} and we have the explicit structure of the KBSM of the connected sum of two solid tori.\qed

\def\Rzero{{R^{[0]}}}
\def\Rloc{{R_{\mathrm{loc}}}}

\subsection{Consequences}
 \begin{corollary}\label{r.cons1}
 
 \begin{enumerate} \hfill
 
 \item The map $\hR \embed \hR[y]$ descends to an embedding $\hR \embed \SH$. 
 
 \item Denote the image of the map in part (1) by $\Rzero$. The quotient $\SH/\Rzero$ is a direct sum of torsion cyclic modules
 \be 
 \frac{\mathcal S(H_1 \ \# \ H_1)}{\Rzero} = \bigoplus \limits _{n=1}^{\infty} \frac{\hat R}{q^{2n+2}-q^{-2n-2}}= \bigoplus \limits _{n=1;n_2, n_3=0}^{\infty} \frac{R}{q^{2n+2}-q^{-2n-2}},
 \ee
 where the summand corresponding to $(n,n_2,n_3)$ is $R/(q^{2n+2}-q^{-2n-2})$.
 \item Let $\Rloc$ be the ring $R$ with the inverses of $\{ n\}$ adjoined for all $n \ge 1$. Then
 \begin{equation*}
     \SH \ot_R \Rloc \cong \Rloc[x_1, x_2]
 \end{equation*} 

This ring $\Rloc$ is the smallest extension of $R$, for which this result holds.

\item For each $n \ge 1$,  $t_n:=G_n/\{1\} \pmod {\fG}$ is a non-zero torsion element of $\SH$. Namely $\{1\} t_n=0$.
 
 \end{enumerate}

 \end{corollary}
 \bpr Since $\hR \cap \fG =\{0\}$ the map $\hR \embed \hR[y]$ descends to an embedding $\hR \embed \SH$, which proves part (1). Parts (2) and (3) follow easily from the form of $G_n$, while part (4) also follows from the explicit form of $G_n$.
 
 \epr

 Note that part (3) of Corollary \ref{r.cons1} is a generalization of a result of the third author \cite{connsum}, which states that $\mathcal S(H_1 \ \# \ H_1; \mathbb Q(q)) \cong \mathcal S(H_1; \mathbb Q(q)) \otimes_{\mathbb Q(q)} \mathcal S(H_1; \mathbb Q(q))$.

\def\Tor{{\mathsf{Tor}}}
\def\hR{\hat R}
\def\fG{\mathcal{G}}
\def\tF{{\mathsf{tF}}}
\def\Frac{{\mathsf{Frac}}}
\def\J{{\mathcal J}}

For a $\Zq$-module $M$ let $\Tor(M)$ be its torsion submodule, and $\tF(M) = M/\Tor(M)$ be the torsion-free quotient.
\bcor (a) The $\Zq$-torsion of $\SH$ is a direct sum of copies of $\Zq/\{1\}$ and hence 
is annihilated by $\{1\}$. More precisely
\be 
\Tor(\SH) \cong \bigoplus _{n=1}^\infty  \bigoplus _{n_1, n_2=0}^\infty \Zq/\{ 1\},
\ee
where the summand corresponding to $(n_1, n_2, n)$ is generated by $\phi(x_1^{n_1} x_2^{n_2} J_n)$. Here
\be J_n:= G_n/\{1\}=  [n+1]S_n(y) +  (-1)^{(n+1)} S_n(x_1)S_n(x_2) \}, \ [k] := \frac{\{ k\} }{\{1\}} \in \Zq.
\ee

(b) The torsion free quotient $\tF(\SH)$ is not a free $\Zq$-module.

(c) Over $\BQ[q^{\pm 1}]$, $\Tor(\SH)$ is a direct summand of $\SH$. Consequently,
\be \SH \cong \Tor(\SH) \oplus \tF(\SH), \ \text{over} \ \BQ[q^{\pm 1}].
\ee
\ecor

\bpr 

 (a) Let $\J$ be the $\hR$-module spanned by $J_n, n\ge 1$. Since $\deg_y J_n =n$, it is easy to see that $\J$ is free over $\hR$-module with basis $\{J_n, n\ge 1\}$. Similarly, $\fG$ is free over $\hR$-module with basis $\{G_n, n\ge 1\}$. It follows that 
\be \J/\fG \cong \bigoplus _{n=1}^\infty\hR\cdot J_n /\{1\} \cong  \bigoplus _{n=1}^\infty  \bigoplus _{n_1, n_2=0}^\infty \Zq/\{ 1\}.
\ee
We have the following exact sequence of $\hR$-modules
\be 
0 \to \J/\fG \to \hR[y]/\fG \to \hR[y] / \J \to 0.
\ee
Let $\Frac(\hR)$ be the  field of fractions of $\hR$.
The $\hR$-module $\hR[y] / \J $ is the result of adjoining $\{\frac{S_n(x_1) S_n(x_2)}{[n+1]}, n \ge 0 \}$, to $\hR$.
 The result is the $\hR$-submodule of $\Frac(\hR)$ spanned by $\Bigl\{\frac{S_n(x_1) S_n(x_2)}{[n+1]}, n \ge 0\Bigr \} $, which is a domain. 
 This shows that $ \J/\fG \cong \Tor(\SH)$.
 
 (b) Let $A= \Frac(\BZ[x_1, x_2]  )[q^{\pm1}]$. Then $\SH \ot_\Zq A$ is the result of adjoining $\Bigl \{\frac{1}{[n]}, n \ge 1 \Bigr \}$ to $A$, which is not a free $A$-module.
 
 (c) Since $\Tor(\SH)$ is annihilated by $\{1\}$, it is of bounded order in the sense of \cite{Kaplansky}. By \cite[Theorem 8 and Section 12]{Kaplansky}, over a principal ideal domain, if the torsion submodule has bounded order, then it is a direct summand.\epr

\bque In $\SH$ is there a  torsion element  annihilated by $z \in R$, which is coprime with $\{1\}$?
\eque

\begin{conjecture}\cite{rheasolomarche, s1s2connsumbkw}

\begin{enumerate} 
    \item The KBSM of any closed, prime, oriented $3$-manifold over $\mathbb Z[A^{\pm 1}]$ can be decomposed into the direct sum of free modules and torsion modules.
    
    \item The KBSM of any non-prime, oriented $3$-manifold over $\mathbb Z[A^{\pm 1}]$ cannot be decomposed into the direct sum of free modules and torsion modules.
    
\end{enumerate}

\end{conjecture}

\newpage


\begin{thebibliography}{999999}

\bibitem[Bak]{rheasolomarche}
R. P. Bakshi, A counterexample to the generalisation of Witten's conjecture, {\it Inverse problems: in memory of Zbigniew Oziewicz}, 73–82.
{\it Contemp. Math.}, 824. e-print: \href{https://arxiv.org/abs/2205.01653}{arXiv:2205.01653} [math.GT].

\bibitem[BKSW]{s1s2connsumbkw}
R. P. Bakshi, S. Kim, S. Shi, X. Wang, On the Kauffman bracket skein module of $(S^1 \times S^2) \ \# \ (S^1 \ \times S^2)$, {\it Journal of Algebra}, 673 (2025), 103–137. \href{https://arxiv.org/abs/2405.04337}{arXiv:2405.04337} [math.GT]. 


\bibitem[BP]{counterhandle}
R. P. Bakshi, J. H. Przytycki, Kauffman bracket skein module of the connected sum of handlebodies: A counterexample, {\it Manuscripta Math.} 167 (2022), no. 3-4, 809–820. \href{https://arxiv.org/abs/2005.07750}{arXiv:2005.07750} [math.GT].





\bibitem[BHMV]{BHMV1}
C. Blanchet, N. Habegger, G. Masbaum, P. Vogel,
Topological quantum field theories derived from the Kauffman bracket. 
{\it Topology} 34 (1995), no. 4,
883–927. 

\bibitem[BW]{bw1}
F. Bonahon, B. Wong, 
Quantum traces for representations of surface groups in $SL_2(\mathbb{C})$.
{\it Geom. Topol.} 15 (2011), no. 3, 1569–1615. \href{https://arxiv.org/abs/1003.5250}{arXiv:1003.5250} [math.GT].


\bibitem[Bul1]{sl2cbullock}
D. Bullock,
Rings of $SL_2(\mathbb{C})$-characters and the Kauffman bracket skein module. {\it Comment. Math. Helv.} 72 (1997), no. 4, 521–542. 

\bibitem[Bul2]{trefoilbullock} 
D. Bullock,
On the Kauffman bracket skein module of surgery on a trefoil. 
{\it Pacific J. Math.} 178 (1997), no. 1, 37–51.

\bibitem[BuLo]{knotext}
D. Bullock, W. F. Lo Faro,
The Kauffman bracket skein module of a twist knot exterior. 
{\it Algebr. Geom. Topol.} 5 (2005), 107–118. \href{https://arxiv.org/abs/math/0402102}{arXiv:math/0402102} [math.QA].

\bibitem[BuPr]{smquant}
D. Bullock, J. H. Przytycki, Multiplicative structure of Kauffman bracket skein module quantizations. {\it Proc. Amer. Math. Soc.} 128 (2000), no. 3, 923–931. \href{https://arxiv.org/abs/math/9902117}{arXiv:math/9902117} [math.QA]. 


\bibitem[CF]{fokchekhov}
L. O. Chekhov, V. V. Fok, 
Quantum Teichmüller spaces. (Russian. Russian summary) 
{\it Teoret. Mat. Fiz.} 120 (1999), no. 3, 511–528; translation in 
{\it Theoret. and Math. Phys.} 120 (1999), no. 3, 1245–1259.


\bibitem[CL]{CL}
F. Constantino, T. T. Q. Lê, Stated skein modules of $3$-manifolds and TQFT, {J. Inst. Math. Jussieu} 24 (2025), no. 3, 663–703. \href{https://arxiv.org/abs/2206.10906}{arXiv:2206.10906} [math.GT].

\bibitem[DM]{pairofpantss1}
M. D\c{a}bkowski, M. Mroczkowski,
KBSM of the product of a disk with two holes and $S^1$.
{\it Topology Appl.} 156 (2009), no. 10, 1831–1849. \href{https://arxiv.org/abs/0808.3782}{arXiv:0808.3782} [math.GT].




\bibitem[EG]{EG}
S. Esebre, R. Gela, A recursive relation in the $(2p+1,2)$ torus knot,  Accepted to {\it Contemporary Mathematics}, e-print: \href{https://arxiv.org/abs/2409.18960}{arXiv:2409.18960} [math.GT]. 


\bibitem[FGL]{FGL}
C. Frohman, R. Gelca, W. F. Lofaro,
The $A$-polynomial from the noncommutative viewpoint. 
{\it Trans. Amer. Math. Soc.} 354 (2002), no. 2, 735–747. \href{https://arxiv.org/abs/math/9812048}{arXiv:math/9812048} [math.QA].



\bibitem[HK]{HK}
J. Hoste, M. Kidwell,
Dichromatic link invariants. 
{\it Trans. Amer. Math. Soc.} 321 (1990), no. 1, 197–229. 

\bibitem[HP]{HP} J. Hoste, J.H. Przytycki,
A survey of skein modules of 3-manifolds;
in  Knots 90, Proceedings of the International Conference on Knot
Theory and Related Topics, Osaka (Japan), August 15-19, 1990, Editor
A.~Kawauchi,Walter de Gruyter 1992, 363-379.

\bibitem[HP1]{kbsmlens}
J. Hoste, J. H. Przytycki,  The $(2,\infty)$-skein module of lens spaces; a generalization of the Jones polynomial. {\it J. Knot Theory Ramifications} 2 (1993), no. 3, 321–333.

\bibitem[HP2]{s1s2}
J. Hoste, J. H. Przytycki, The Kauffman bracket skein module of $S^1 \times S^2$. {\it Math. Z.} 220 (1995), no. 1, 65–73.

\bibitem [HP3]{whitehead}
J. Hoste, J. H. Przytycki,
The $(2,\infty)$-skein module of Whitehead manifolds. 
{\it J. Knot Theory Ramifications} 4 (1995), no. 3, 411-427.

\bibitem[Kap]{Kaplansky} I. Kaplansky, {\it Infinite abelian groups}, Univ. Michigan Press, Ann Arbor, MI, 1954; MR0065561.

\bibitem[Kas]{kashaev}
R. M. Kashaev, 
Quantization of Teichmüller spaces and the quantum dilogarithm.
{\it Lett. Math. Phys.} 43 (1998), no. 2, 105–115. \href{https://arxiv.org/abs/q-alg/9705021}{arXiv:q-alg/9705021} [math.QA].

\bibitem[Kau]{tlnkauffman} 
L. H. Kauffman, State models and the Jones polynomial, {\it Topology}, 26, 1987, 395-407.



\bibitem[Le1]{2bk}
T. T. Q. Lê, The colored Jones polynomial and the A-polynomial of knots. {\it Adv. Math.} 207 (2006), no. 2, 782–804. \href{https://arxiv.org/abs/math/0407521}{arXiv:math/0407521} [math.GT].

\bibitem[Le2]{lequantum}
T. T. Q. Lê,
Quantum Teichmüller spaces and quantum trace map.
{\it J. Inst. Math. Jussieu} 18 (2019), no. 2, 249–291. \href{https://arxiv.org/abs/1511.06054}{arXiv:1511.06054} [math.GT].

\bibitem[LT]{2bl}
T. T. Q. Lê, A. Tran,
The Kauffman bracket skein module of two-bridge links. 
{\it Proc. Amer. Math. Soc.} 142 (2014), no. 3, 1045–1056. \href{https://arxiv.org/abs/1111.0332} {arXiv:1111.0332} [math.GT].



\bibitem[Mul]{mullerskein}
G. Muller,
Skein and cluster algebras of marked surfaces.
{\it Quantum Topol.} 7 (2016), no. 3, 435–503. \href{https://arxiv.org/abs/1204.0020}{arXiv:1204.0020} [math.QA].

\bibitem[Mro1]{prism}
M. Mroczkowski, Kauffman bracket skein module of a family of prism manifolds.
{\it J. Knot Theory Ramifications} 20 (2011), no. 1, 159–170.

\bibitem[Mro2]{rp3rp3}
M. Mroczkowski, Kauffman bracket skein module of the connected sum of two projective spaces. 
{\it J. Knot Theory Ramifications} 20 (2011), no. 5, 651–675. \href{https://arxiv.org/abs/1008.1007} {arXiv:1008.1007} [math.GT].



\bibitem[Prz1]{smof3}
J. H. Przytycki, Skein modules of 3-manifolds. {\it Bull. Polish Acad. Sci. Math.} 39 (1991), no. 1-2, 91–100. \href{https://arxiv.org/abs/math/0611797}{arXiv:math/0611797} [math.GT].


\bibitem[Prz2]{fundamentals}  
J. H. Przytycki, Fundamentals of Kauffman bracket skein modules. {\it Kobe Math. J.}, 16(1), 1999, 45-66. \href{https://arxiv.org/abs/math/9809113}{arXiv:math/9809113} [math.GT]. 

\bibitem[Prz3]{connsum}
J. H. Przytycki,
Kauffman bracket skein module of a connected sum of $3$-manifolds.  
{\it Manuscripta Math.} 101 (2000), no. 2, 199–207. \href{https://arxiv.org/abs/math/9911120}{arXiv:math/9911120} [math.GT].

\bibitem[PS]{ps1}
J. H. Przytycki, A. S. Sikora, On skein algebras and $Sl_2(\mathbb{C})$-character varieties. {\it Topology} 39 (2000), no. 1, 115–148. \href{https://arxiv.org/abs/q-alg/9705011} {arXiv:q-alg/9705011} [math.QA].


\bibitem[Tur1]{turaevsolidtorus}
V. G. Turaev,
The Conway and Kauffman modules of a solid torus.
{\it Zap. Nauchn. Sem. Leningrad. Otdel. Mat. Inst. Steklov.} (LOMI) 167 (1988), Issled. Topol. 6, 79–89, 190; translation in {\it J. Soviet Math.} 52 (1990), no. 1, 2799–2805. 

\bibitem[Tur2]{Turaev1}
V. Turaev, Skein quantization of Poisson algebras of loops on surfaces, {\it Ann. Sci. Sc. Norm. Sup.}
(4) 24 (1991), No. 6, 635–704.


\end{thebibliography}
\end{document}